\documentclass[11pt, twoside]{amsart}
\usepackage{amsmath,mathtools}
\usepackage{amsmath, amsthm, amssymb, amsfonts, enumerate}
\usepackage[colorlinks=true,linkcolor=blue,urlcolor=blue,citecolor=blue]{hyperref}
\usepackage{dsfont}
\usepackage{color}
\usepackage{todonotes}
\usepackage{epstopdf}
\usepackage{bbm}

\usepackage[utf8]{inputenc}
\usepackage{bm}
\usepackage{amsfonts}
\usepackage{amsfonts}
\usepackage{textcomp}
\usepackage{amssymb}
\usepackage{float}
\usepackage{tikz}
\usepackage{epsfig}
\usepackage{amsmath}
\usepackage[english]{babel}
\usepackage{a4}
\usepackage{enumerate}
\usepackage{tcolorbox}
\usepackage{soul}

\usepackage{geometry}
\newcommand{\stkout}[1]{\ifmmode\text{\sout{\ensuremath{#1}}}\else\sout{#1}\fi}

\usepackage[T1]{fontenc}

\newtheorem{theorem}{Theorem}[section]

\newtheorem{remark}[theorem]{Remark}
\newtheorem{assumption}[theorem]{Assumption}
\newtheorem{lemma}[theorem]{Lemma}

\newtheorem{definition}[theorem]{Definition}

\newtheorem{example}[theorem]{Example}

\newcommand{\ds}{\displaystyle}

\def \E{\mathsf{E}}

\def \P{\mathsf{P}}
\def \R{\mathbb{R}}

\def\d{\mathrm{d}}

\definecolor{red}{rgb}{1.0,0.0,0.0}

\definecolor{blu}{rgb}{0.0,0.0,1.0}

\definecolor{gre}{rgb}{0.03,0.50,0.03}

\usepackage[T1]{fontenc}

\usepackage{amsmath,mathtools}
\usepackage{amsmath, amsthm, amssymb, amsfonts, enumerate}

\usepackage[colorlinks=true,linkcolor=blue,urlcolor=blue]{hyperref}
\usepackage{dsfont}
\usepackage{color}
\usepackage{geometry}
\usepackage{todonotes}
\usepackage{epstopdf}
\newtheorem{proposition}[theorem]{Proposition}

\def \E{\mathsf{E}}
\def \P{\mathsf{P}}

\def \R{\mathbb{R}}

\def\d{\mathrm{d}}

\definecolor{red}{rgb}{1.0,0.0,0.0}

\definecolor{blu}{rgb}{0.0,0.0,1.0}

\definecolor{gre}{rgb}{0.03,0.50,0.03}

\definecolor{darkviolet}{rgb}{0.58, 0.0, 0.83}

\numberwithin{equation}{section}

\def\sqr#1#2{{\vcenter{\vbox{\hrule height .#2pt \hbox{\vrule
 width .#2pt height#1pt \kern#1pt \vrule
width .#2pt} \hrule height .#2pt}}}}

\def\ds{\begin{displaystyle}}
\def\eds{\end{displaystyle}}

\def\<{\left\langle }
\def\>{\right\rangle }

\def\R{\mathbb R}
\def\N{\mathbb N}

\def\1{\mathbf 1}
\def\to{\rightarrow}

\title{On Mean Field Games in Infinite Dimension}
\vspace{-1.5truecm}

\author[Federico]{Salvatore Federico}
\author[Gozzi]{Fausto Gozzi}
\author[{\'{S}}wi{\k{e}}ch]{Andrzej {\'{S}}wi{\k{e}}ch}

\address{S.~Federico: Dipartimento di Matematica, Universit\`a ``Alma Mater'' di Bologna, Piazza di Porta San Donato, Bologna, Italy}
\email{\href{mailto:s.federico@unibo.it}{s.federico@unibo.it}}
\address{F.~Gozzi: Dipartimento di Economia e Finanza, LUISS University, Viale Romania 32, Rome, Italy}
\email{\href{mailto:fgozzi@luiss.it}{fgozzi@luiss.it}}
\address{A.~{\'{S}}wi{\k{e}}ch: School of Mathematics, Georgia Institute of Technology, 
Atlanta, GA 30332, USA}
\email{\href{mailto:swiech@math.gatech.edu}{swiech@math.gatech.edu}}

\date{\today}

\date{\today}
\begin{document}
\maketitle

\begin{abstract}
{
%
We study a Mean Field Games (MFG) system in a real, separable infinite dimensional Hilbert space. The system consists of a second order parabolic type equation, called Hamilton-Jacobi-Bellman (HJB) equation in the paper, coupled with a nonlinear Fokker-Planck (FP) equation. Both equations contain a Kolmogorov operator. Solutions to the HJB equation are interpreted in the mild solution sense and solutions to the FP equation are interpreted in an appropriate weak sense. We prove well-posedness of the considered MFG system under certain conditions. The existence of a solution to the MFG system is proved using Tikhonov's fixed point theorem in a proper space. Uniqueness of solutions is obtained under typical separability and Lasry-Lions type monotonicity conditions.
}

\end{abstract}


\medskip





%
%
%
%
%
%
%
%
%
%
%
%
%

\section{Introduction}\label{SE:INTRO}

In this paper, we are concerned  with a class of Mean Field Games (MFG) problems in a real separable Hilbert space $(H,\langle\cdot,\cdot\rangle)$. The analytical framework is the following. Let $\mathcal{P}_{1}(H)$ denote the space of Borel probability measures on $H$ having finite first moment and let $T>0$. We consider the following system of coupled Hamilton-Jacobi-Bellman (HJB) and nonlinear Fokker-Planck (FP) equations for  functions $v:[0,T]\times H\to \mathbb{R}$ and $m:[0,T]\to \mathcal{P}_1(H)$,
\begin{align}
- {\partial_t}v(t,x)- {L} v (t,x) +\mathcal{H}(x, Dv(t,x),m(t)) =0, \ \ \ v(T,\cdot)=G(\cdot,m(T)), \ \ \ \ \ \label{HJBbis}\tag{\textbf{HJB}}
\\\nonumber\\
\partial_{t}m(t) - {L}^{*} m(t)-\mbox{div}\Big(\mathcal{H}_p(x,Dv(t,x),m(t))\,m(t)\Big)=0,  \ \ \ m(0)=m_{0}, \ \ \ \ \ \ \label{FPbis}\tag{\textbf{FP}}
\end{align}
where
\begin{enumerate}[(i)]
\item $m_{0}\in\mathcal{P}_{1}(H), \ G:H\times \mathcal{P}_{1}(H)\to\R$;
\item $L$ is the operator formally defined on functions $\phi:H\to \R$ by
$$L\phi(x)= \langle Ax , D\phi(x)\rangle + \frac{1}{2}\mbox{Tr}[D^2 \phi(x)],$$
where
$A:D(A)\subseteq H\to H$ is a closed and densely defined linear operator;
{\item $L^{*}$ is, formally, the adjoint of $L$ and the operator $\mathbf{div}$ is, formally, the opposite of the adjoint of the gradient: both are used here only for notational convenience as the rigorous definition of a solution to \eqref{FPbis} will not involve them, see Section \ref{sec:FP} below;}
\item $\mathcal{H}:H\times H\times\mathcal{P}_{1}(H)\to\R$   and $\mathcal{H}_{p}$ denotes the  Fr\'echet derivative of $\mathcal{H}$ with respect to the second argument.
\end{enumerate}
{We point out that we slightly abuse terminology here, since the first equation is not necessarily a Hamilton-Jacobi-Bellman equation as the function $\mathcal{H}$ can be very general. However, since our primary motivation comes from this case, we will call it an HJB equation.}

The theory of MFG originated from the pioneering works of Lasry and Lions,  and of Huang, Caines, and Malhamé (see \cite{huang2007large, LasryLionsCRMI06, LasryLionsCRMII06, LasryLionsJJM07}). This theory was primarily developed to provide an analytical framework for (stochastic) differential games involving a large number of symmetric players. In the limit, as the number of players approaches infinity and assuming that the actions of individual players do not directly influence the overall population dynamics, the concept of Nash equilibrium leads to the study of a coupled system of PDEs of the type given by equations \eqref{HJBbis} and \eqref{FPbis}.
Since the inception of this theory, substantial progress has been made in finite-dimensional settings, both theoretically and practically. Notable references for this research area include the book  \cite{carmona2018probabilistic} and the survey paper \cite{MR4214774}. 
However, contributions to the theory of Mean Field Games (MFG) in infinite-dimensional spaces remain surprisingly limited. To our knowledge, the literature consists of only a few essential papers, such as \cite{fouque2018mean}, which addresses systemic risk by taking the limit of the $N$-player model in \cite{carmona2018systemic}, and recent works \cite{federico2024linear,liu2024hilbert}, both of which focus on the linear-quadratic case. This scarcity is somewhat in contrast with the established theory of stochastic optimal control in infinite-dimensional spaces and its connection with Hamilton-Jacobi-Bellman (HJB) equations, which has been well developed over several decades, beginning with foundational works by Barbu and Da Prato (see, e.g., \cite{BarbuDaPrato83book}). For a general overview of this subject and its applications, we refer to the recent monograph \cite{FGS}, and see also \cite{DaPratoZabczyk02,Nisio}.

In this paper, we aim to bridge this gap by providing a preliminary contribution to the study of MFG in infinite dimensions beyond the linear-quadratic case. Building upon existing results for both HJB and Fokker-Planck equations in infinite-dimensional settings, we investigate the nonlinear coupled system of MFG equations \eqref{HJBbis}--\eqref{FPbis}. Our main contribution is the proof of well-posedness of this system under certain conditions. The HJB equation \eqref{HJBbis} is interpreted in the mild form, allowing us to use well-known theories of mild solutions to partial differential equations in Hilbert spaces. However, this approach requires stronger assumptions on the operator $A$ to ensure certain smoothing properties of the associated Ornstein-Uhlenbeck transition semigroup.

The Fokker-Planck equation \eqref{FPbis} is interpreted in a weak form using a special class of test functions. Existence of a solution to the system \eqref{HJBbis}--\eqref{FPbis} is established with a typical MFG fixed-point argument via Tikhonov's fixed-point theorem, which requires the existence of a unique weak solution to the linear \eqref{FPbis} when the term $\mathcal{H}_p(x,Dv(t,x),m(t))$ is fixed. Fokker-Planck-type equations in Hilbert spaces have been widely studied, with various results available on existence and uniqueness (see \cite{BDaPR.2010,bogachev2007uniqueness,BDPR,bogachev2009fokker, DaPFR.2013,RZZ.2014,Wie.2013}). To show that our linear \eqref{FPbis} has a weak solution, we follow a common approach of linking such an equation to a stochastic differential equation (SDE) in $H$ and noting that the laws of solutions to this SDE are weak solutions of the linear \eqref{FPbis}. We apply results from \cite{DFPR} on SDEs in Hilbert spaces with bounded measurable drift, adapting them to our setting. Uniqueness of weak solutions then follows easily from a result in \cite{BDPR}, chosen here for its clarity and accessibility. However, this approach imposes additional restrictions on the operator $A$, ultimately leading to fairly strict assumptions on $A$.

Uniqueness of solutions to the MFG system \eqref{HJBbis}--\eqref{FPbis} is obtained under typical monotonicity and separability conditions on $\mathcal{H}$ (see \cite{MR4214774} and \cite{LasryLionsJJM07}). To carry out the uniqueness argument, which is now standard in the finite-dimensional case, we approximate the mild solution of the Kolmogorov equation obtained by subtracting two different solutions to the \eqref{HJBbis} equations in the \eqref{HJBbis}--\eqref{FPbis} system by constructing smooth solutions of approximating Kolmogorov equations that belong to the class of test functions used in the definition of a weak solution for the Fokker-Planck equation. Since this is the first study on the subject, we chose not to address the most general case to avoid obscuring the main ideas with excessive technical detail. A possible generalization could involve a more flexible operator $L$, in which the term $\mbox{Tr}[D^2 \phi(x)]$, associated with the cylindrical Wiener process $W$ in the underlying SDE, is replaced by $\mbox{Tr}[QD^2 \phi(x)]$ for some self-adjoint operator $Q$ (corresponding to a more general additive noise in the SDE). This could potentially relax the assumptions on $A$. Another interesting direction would be to interpret equation \eqref{HJBbis} in the viscosity sense, which could also include the case of first-order MFG. We plan to investigate this in future work.

Regarding possible applications, it would be interesting to try to relax the assumptions on the operator $A$ to cover potentially interesting cases. In particular, it would be interesting to treat the case of first order operators $A$, some of which appear in financial and economic applications, such as  models of systemic risk with delays (as in \cite{fouque2018mean}) or models of 
optimal investment with vintage capital (see  \cite{FAGGIAN2021102516, FEICHTINGER2006143} in the case of a single firm with vintage structure; \cite{calvia2024mean} in the case of a continuum of homogenous firms but without vintage structure).

The plan of the paper is the following. In Section \ref{SE:RM}, we introduce the basic notation and assumptions, presenting preliminary technical results on the well-posedness of the SDE related to the (linear) Fokker-Planck equation, along with some estimates for its solutions. Section \ref{sec:FP} is dedicated to the existence and uniqueness of weak solutions for the Fokker-Planck equation \eqref{FP}. In Section \ref{sec:HJB}, we define the notion of a mild solution to \eqref{HJBbis} and collect some foundational results on mild solutions. In Section \ref{sec:MFGsystem}, we show the existence of solutions to the MFG system \eqref{HJBbis}--\eqref{FPbis}, while the uniqueness of solutions is established in Section \ref{sec:exist-uniq}. An example that satisfies the manuscript’s assumptions is discussed in Section \ref{sec:example}. Finally, in Appendix \ref{app:A}, we present a result on the compactness of sets in spaces of probability measures on a Hilbert space, and in Appendix \ref{app:B}, we discuss essential concepts and results related to Kolmogorov equations in Hilbert spaces that are needed in this paper.

\section{Preliminaries and Assumptions}\label{SE:RM}


{We recall that $H$ is a real separable Hilbert space with  inner product $\langle\cdot,\cdot\rangle$. We will write $|\cdot|$ for the norm in $H$; if needed we stress $H$ in the notation writing $|\cdot|_{H}$. We always identify $H$ with its dual space. For $R>0$, we denote $B_R:=\{x\in H: |x|\leq R\}$. Throughout the paper,  $\mathbb{N}$ will denote the set of natural numbers without $0$.}

\subsection{Basic spaces and notations}
We introduce the notation for various functional and operator spaces used in the paper. Below, by  $I$ we denote an interval in $\R$. {Given a function $u:I\times H\to \R,(t,x)\mapsto u(t,x)$, we will write $\partial_t u$ for the partial derivative of $u$ with respect to $t$ and $Du,D^2u$, respectively, for the first and second order Fr\'echet derivatives of $u$ with respect to the $x$-variable.}
For a bounded function between normed spaces, we denote its sup-norm in the usual way by $\|\cdot\|_{\infty}$.

Below is a list of basic spaces used in the paper.
\begin{enumerate}[(i)]
\item $\mathcal{L}(H)$ is the space of bounded linear operators from $H$ to $H$.
\medskip
\item $\mathcal{L}^{+}(H)$ is the space of bounded linear self-adjoint nonnegative operators from $H$ to $H$.
\medskip
\item $\mathcal{L}_{1}(H)$ is the space of bounded linear trace-class operators from $H$ to $H$.
\medskip
\item $\mathcal{L}_{1}^{+}(H)=\mathcal{L}^{+}(H)\cap \mathcal{L}_{1}(H)$.
\medskip
\item $C_{b}(H)$ is the space of continuous and bounded functions $u:H\to\R$.
\medskip
\item $C_{b}(I\times H)$ is the space of  continuous and bounded functions $u:I\times H\to\R$.
\medskip
\item  $C_{b}(I\times H;H)$ is the space of continuous and bounded functions $u:I\times H\to H$.
\medskip
\item  $C_{b}(I\times H;\mathcal{L}(H))$ (respectively,  $C_{b}(I\times H;\mathcal{L}_1(H))$)  is the space of continuous and bounded functions $u:I\times H\to \mathcal{L}(H)$ (respectively, $u:I\times H\to \mathcal{L}_{1}(H)$).
\medskip
\item  {$C^{0,1}_{b}(I\times H)$ is the space of continuous and bounded functions $u:I\times H\to\R$,  which are Fr\'echet differentiable with respect to the variable $x$ and such that $Du\in C_b(I\times H;H)$.}
\medskip
{\item $UC_b(H)$ (respectively, $UC_b(H;H)$, $UC_b(H;\mathcal{L}_{1}(H))$) is the space of bounded and uniformly continuous functions $u:H\to\R$ (respectively, $u:H\to H$, $u:H\to \mathcal{L}_{1}(H)$).
\medskip
\item $UC_b^2(H)=\{u:H\to\R: \ u\in UC_b(H),Du\in UC_b(H;H),D^2u\in UC_b(H;\mathcal{L}(H))\}$.
\medskip
\item  Finally, 
\begin{align*}
C_b^{1,2}(I\times H)=&\big\{u\in C_{b}(I\times H): \  \partial_tu\in C_{b}(I\times H), \ Du\in C_b(I\times H;H),\\& \ \ \ \ D^2u\in C_b(I\times H;\mathcal{L}(H))\big\}.
\end{align*}
}
\end{enumerate}

We denote by $\mathcal{M}(H)$ the topological vector space of finite signed measures on $H$ endowed with the weak topology, i.e., the locally convex topology on $\mathcal{M}(H)$ induced by the family of seminorms
\begin{equation}\label{semi}
\left|\mu\right|_{f}:=\left|\int_{H} f(x) \mu(\d x)\right|, \ \ \ \ f\in C_{b}(H).
\end{equation}
By $\mathcal{P}(H)$ we denote the space of Borel probability measures on $H$. The weak topology defined on it is the one inherited by the weak topology of $\mathcal{M}(H)$ defined above.
We denote by $\mathcal{P}_1(H)$ the subset of $\mathcal{P}(H)$ of the probability measures having finite first moment.
{In $\mathcal{P}_1(H)$ we have the Monge-Kantorovich distance (see, e.g., \cite{AGS,Villani})
 \begin{equation}\label{definition_wasserstein_distance}
        \mathbf{d}_{1}(\mu,\nu) = \inf_{\gamma\in \Gamma(\mu,\nu)} \int_{H\times H} |x-y| \gamma(\mathrm{d}x,\mathrm{d}y),
    \end{equation}
    where $\Gamma(\mu,\nu)$ is the set of all Borel probability measures on $H \times H$ with first and second marginals $\mu$ and $\nu$, respectively. Unless differently specified, the space $\mathcal{P}_{1}(H)$ will be considered endowed with the topology induced by $\mathbf{d}_{1}$ which makes it a complete metric space. We notice that, if $(\Omega,\mathcal{F},\P)$ is a probability space and $X,Y\in L^1(\Omega;H)$ are such that  $\mathcal{L}(X)= \mu$,  $\mathcal{L}(Y) = \nu$, then $\gamma=(X\times Y)_{\texttt{\#}}\P\in \Gamma(\mu,\nu)$, so that
    \[
    \mathbf{d}_{1}(\mu,\nu)\leq  \int_{H\times H} |x-y| \,\,\gamma(\mathrm{d}x,\mathrm{d}y)=\int_{\Omega} |X(\omega)-Y(\omega)|\mathrm{d}\P(\omega)=\E[|X-Y|].
    \]}
We will deal with the space
\begin{equation}
\label{S}
\mathcal{S}=C\big([0,T]; (\mathcal{P}_1(H),\mathbf{d}_{1})\big).
\end{equation}
It is a complete metric space endowed with the metric
\[
\rho_\infty(\mu,\nu)=\sup_{t\in[0,T]} \mathbf{d}_{1}(\mu(t),\nu(t)).
\]
We will also use the symbol $\rho_\infty$ to denote the induced metric topology.

\subsection{Weighted spaces}

Given any $\gamma >0$, we define
\begin{multline}
\label{eqB:Cmgammadefbis}
 C_{b,\gamma }\left( [0,T) \times H \right)
: =
\big\{ f\colon [0,T)\times H \to \R \hbox{ s.t. }
f \in {C}_{b} ([0,\tau] \times H) \  \forall \tau \in (0,T)
\\
\; \text{and }
(t,x) \mapsto (T-t)^{\gamma}f(t,x) \hbox{ belongs to } {C}_{b}([0,T) \times H) \big\}.
\end{multline}
Similarly, we define
 \begin{multline}
\label{eqB:Cmgammadefbis1}
 C_{b,\gamma }\left( [0,T) \times H; H \right)
: =
\big\{ w\colon [0,T)\times H \to H \hbox{ s.t. }
f \in {C}_{b} ([0,\tau] \times H;H) \  \forall \tau \in (0,T)
\\
\; \text{and }
(t,x) \mapsto (T-t)^{\gamma}f(t,x) \hbox{ belongs to } {C}_{b}([0,T) \times H;H) \big\}.
\end{multline}
The spaces above are Banach spaces when endowed with the norm (see e.g.
\cite{Gozzi95,Gozzi02} and \cite{Cerrai95})
$$
\|f\|_{{C_{b,\gamma }\left([0,T) \times H \right)}}
:=
\sup_{(t,x)\in [0,T)\times H}
(T-t)^{\gamma}\left\vert f(t,x)\right\vert, 
$$
$$ \|f\|_{{C_{b,\gamma }\left([0,T) \times H;H \right)}}
:=
\sup_{(t,x)\in [0,T)\times H}
(T-t)^{\gamma}\left\vert f(t,x)\right\vert_{H}.
$$
We will often simply write $\|f\|_{{C_{b,\gamma }}}$. Moreover, we define
\begin{multline}
\label{eq4:Cmgammadefbis}
 C^{0,1}_{b,\gamma }\left( [0,T] \times H\right)
: =
\left\{ f \in C_b([0,T]\times H) :
f \in {C}^{0,1}_{b} ([0,\tau] \times H) \; \forall \tau \in (0,T)
\right.\\
\left.
\; \text{and }
(t,x) \mapsto (T-t)^{\gamma}Df(t,x) \hbox{ belongs to } {C}_{b}([0,T) \times H;H) \right\}.
\end{multline}
The space above is a Banach space when endowed with the norm (see e.g.
\cite{Gozzi95,Gozzi02} and \cite{Cerrai95})
\begin{equation*}
\|f\|_{{C^{0,1}_{b,\gamma }\left( [0,T] \times H \right)}}
:=\|f\|_{\infty}+
\sup_{(t,x)\in [0,T)\times H}
(T-t)^{\gamma}\left\vert Df(t,x)\right|_{H}.
\end{equation*}
We will sometimes simply write $\|f\|_{{C^{0,1}_{b,\gamma }}}$.


%
\subsection{The operator $A$ and the generated semigroup}\label{sub:A}
Concerning the operator $A$, we assume the following.

\begin{assumption}\label{ass:A}
{$A:D(A)\subseteq H \to H$ is closed, densely defined, negative and self-adjoint.} Moreover there exists $\delta>0$ such that $(-A)^{-1+\delta}\in\mathcal{L}_{1}^{+}(H)$.
\end{assumption}
Given Assumption \ref{ass:A},
$A^{-1}$ is compact and self-adjoint. Therefore,
 there exist an orthonormal basis $\{\mathbf{e}_{k}\}_{k\in\N}$ and a decreasing sequence $(\lambda_{k})_{k\in\N}\subset(-\infty,0)$ such that $\lambda _{n}\to -\infty$, with respect to which $A$ is diagonal:
$$
A\mathbf{e}_{k}=\lambda_{k}\mathbf{e}_{k}, \ \ \ k\in\N,
$$
and

\begin{equation}\label{condo}
\sum_{k\in\N} |\lambda_{k}|^{-1+\delta}<\infty.
\end{equation}
Moreover, $A$ generates an analytic $C_{0}$-semigroup of contractions $\{e^{tA}\}_{t\geq 0}\subseteq \mathcal{L}(H)$ and
\begin{equation}\label{eq:sem}
e^{tA}\mathbf{e}_{k}=e^{\lambda_{k}t}\mathbf{e}_{k}, \ \ \  k\in\N.
\end{equation}
%

 \begin{remark}\label{rem:assA}
 \begin{enumerate}[(i)]
 \medskip
 \item[]
 \medskip
\item An example of $A$ that satisfies the above Assumption \ref{ass:A} is the Laplace operator $\partial_{xx}$ in $H=L^{2}([0,1])$ with zero Dirichlet boundary condition.
 \medskip
 \item It can be verified from \eqref{condo} that Assumption \ref{ass:A} implies typical conditions on the operator $A$ ensuring existence and uniqueness of mild solutions of equations in infinite dimensional spaces (see \cite[Ch.\,4]{FGS}), that is:
 \medskip
 \begin{enumerate}[(a)]
 \item One has
 $$
 Q_{t}:=\int_{0}^{t} e^{sA} e^{sA^*}\,\d s\in\mathcal{L}_{1}^{+}(H) \ \ \ \forall t>0
 $$
and there exists $\alpha>0$ such that
 $$
  Q_{t}^{(\alpha)}:=\int_{0}^{t} s^{-\alpha}e^{sA} e^{sA^*}\,\d s \in \mathcal{L}_{1}^{+}(H) \ \ \ \forall t>0.
 $$
 \item
$e^{tA}(H)\subset Q_{t}^{1/2}(H)$ for every $t>0$ and
there exist $C_{0}>0$ and $\gamma\in (0,1)$ such that
$$
|Q_{t}^{-1/2}e^{tA}| \le C_0 t^{-\gamma}, \qquad \forall t \geq 0,
$$
where $Q_{t}^{-1/2}$ is the pseudoinverse of $Q_{t}^{1/2}$ (see \cite[Def.\,B.1]{FGS}).
\end{enumerate}
\end{enumerate}
 \end{remark}

\subsection{The probabilistic framework and the reference SDE}\label{sub:prob}
We now define a probabilistic framework.
We consider a complete filtered probability space $(\Omega,\mathcal{F},(\mathcal{F}_{t})_{t\geq 0},\P)$ supporting a sequence $\{\beta_{k}(t)\}_{k\in\mathbb{N}}$ of independent one dimensional Brownian motions.
 Given the orthonormal basis $\{\mathbf{e}_{k}\}_{k\in\mathbb{N}}$ of Assumption \ref{ass:A},
we  consider the cylindrical Wiener process $\{W_{t}\}_{t\geq 0}$, formally written as
$$W_{t}=\sum_{k\in\mathbb{N}}\beta_{k}(t)\mathbf{e}_{k}.$$
Under Assumption \ref{ass:A},  the stochastic convolution (see Remark \ref{rem:assA}(ii) and \cite[Ch.\,5]{DaPratoZabczyk14})
\begin{equation}\label{stconv}
W_{t}^{A}:=\int_{0}^{t} e^{(t-s)A}\d W_{s}, \qquad t \ge 0,
\end{equation}
 is a well-defined centered  Gaussian process in the Hilbert space $H$ with covariance operator
 $$
 Q_{t}:=\int_{0}^{t} e^{2sA}\,\d s \in \mathcal{L}_{1}^{+}(H)
 $$
%
%
admitting a continuous version.
The process  $$X_{t}^{0,x}:=e^{tA}x+ W_{t}^{A}, \ \ \ t\geq 0,$$
 can be seen as a generalized (called \emph{mild}) solution to the SDE
 $$
\d X_{t}=AX_t\,\d t+\d W_{t}, \ \ \ X_{0}=x\in H.
 $$
$X_{t}^{0,x}$ is called the Ornstein-Uhlenbeck process and is associated to the Markov transition semigroup
\begin{equation}\label{semigroup}
[R_{t}\varphi](x)=\E\left[\varphi (X_{t}^{0,x})\right]=\int_{H}\varphi(y) \mathcal{N}(e^{tA}x,Q_{t})(\d y), \ \ \ \varphi\in C_b(H),
\end{equation}
{called the Ornstein-Uhlenbeck semigroup. Above, $\mathcal{N}(a,Q)$ is the Gaussian measure in $H$ with mean $a\in H$ and covariance operator $Q\in L^{1}_{+}(H)$ (see, e.g.,  \cite[Ch.2, Sec.3]{DaPratoZabczyk02}).}
We will use it later to define the concept of a mild solution to \eqref{HJBbis}.
\bigskip

To construct a weak solution to a linear Fokker-Planck equation we will consider the SDE
\begin{equation}\label{MKV}
\d X_t=   A X_{t}\,\d t+ w(t,X_{t})\,\d t+\d W_{t}, \ \ \ \ \ \mathcal{L}(X_{0})=m_{0}.
\end{equation}
We will need the following assumption. 
\begin{assumption}\label{ass:w}
\begin{enumerate}[(i)]
\item
 $\displaystyle{\int_{H}|x|^4m_0(\d x)<\infty}$.\medskip
\item
$w\in C_b((0,T)\times H; H)$.
 \end{enumerate}
 \end{assumption}

  \begin{theorem}\label{th:mildSDE}
Let $m_{0}\in\mathcal{P}_{1}(H)$ and let Assumptions \ref{ass:A} and  \ref{ass:w}(ii) hold.  There exists a unique weak-mild solution to  \eqref{MKV} in the following sense.
\begin{enumerate}[(i)]
\item  There exist a filtered probability space $(\Omega, \mathcal{F}, (\mathcal{F}_{t})_{t\in[0,T]}, \P)$, an $\mathcal{F}_{0}-$measurable  random variable $X_{0}$ with $\mathcal{L}(X_{0})=m_{0}$,  a cylindrical Wiener process $(W_{t})_{t\in[0,T]}$ with respect to $(\mathcal{F}_{t})_{t\in[0,T]}$, and an adapted process $(X^{w}_{t})_{t\in[0,T]}$  such that
$$
X_{t}^{w}=e^{tA}X_{0}+\int_{0}^{t} e^{(t-s)A} w(s,X_{s}^{w})\d s + W^{A}_{t}, \ \ \ \forall t\geq 0;
$$
\item If another weak-mild solution $\widehat{X}^{w}$ exists, then $\mathcal{L}(X^{w}_{t})=\mathcal{L}(\widehat{X}^{w}_{t})$ for every $t\in[0,T]$.
\end{enumerate}
\end{theorem}
\begin{proof}
When $w$ does not depend on $t$, the result is proved \cite[Th.\,13]{DFPR} and the proof works in the same way if $w$ depends on $t$. We just need to explain how our setup fits into the framework of \cite{DFPR}.

Given our filtered probability space $(\Omega,\mathcal{F},(\mathcal{F}_{t})_{t\geq 0},\P)$ with the Wiener process $W$,
let $(\Omega_1,\mathcal{F}_1,\P_1)$ be another complete atomless probability space.
For any any $m_0\in \mathcal{P}_{1}(H)$, since $(\Omega_1,\mathcal{F}_1,\P_1)$ is nonatomic, there exists an $\mathcal{F}_1$-measurable random variable $\xi_{0}:\Omega_1\to X$ such that $m_0=\mathcal{L}(\xi_0)$ (see, e.g., \cite[Lemma\,5.29]{CarmonaDelarueBook1}).
We define a new probability space $(\Omega_2,\mathcal{F}_2,\P_2)$ as follows:
$$\Omega_2=\Omega\times\Omega_1, \ \  \P_2=\P\otimes\P_1 \ \mbox{on} \  \mathcal{F}\otimes \mathcal{F}_1,$$
$$\mathcal{F}_2= \mbox{completion of} \  \mathcal{F}\otimes \mathcal{F}_1\ \mbox{with respect to} \
\P_2=\P\otimes\P_1.$$ We then define $(\mathcal{F}^2_t)$ to be the augmentation of $\mathcal{F}_t\otimes \mathcal{F}_1$.
 The cylindrical Wiener process $\{W_{t}\}_{t\in[0,T]}$ and the random variable $\xi_{0}$ extend naturally to the space $(\Omega_2,\mathcal{F}_2,(\mathcal{F}^2_t)_{t\in[0,T]},\P_2)$. We denote them by $W^2$ and $X_0$ respectively. The random variable $X_{0}:\Omega\to H$ is $\mathcal{F}^2_0$-measurable, $m_0=\mathcal{L}(X_0)$, and $\{W^2_{t}\}_{t\in[0,T]}$ is a cylindrical Wiener process in the filtered probability space $(\Omega_2,\mathcal{F}_2,(\mathcal{F}^2_t)_{t\in[0,T]},\P_2)$.

Now, following the arguments of \cite{DFPR}, the weak-mild solution is constructed starting from the space $(\Omega_2,\mathcal{F}_2,(\mathcal{F}^2_t)_{t\in[0,T]},\P_2)$ using the Girsanov Theorem by introducing a new cylindrical Wiener process $\{\tilde W^2_{t}\}_{t\in[0,T]}$ and replacing $\P_2$ by $\tilde{\P}_2$, where
\[
\d\tilde{\P}_2={M_T} \d\P_2,
\]
and $(M_s)_{s\in[0,T]}$ is an appropriate martingale such that $M_{0}\equiv 1$. If $f\in C_b(H)$, we then have
\[
\tilde{\E}\left[f(X_0)\right]=\E\left[M_Tf(X_0)\right]=\E\left[\E\left[M_Tf(X_0)\,|\,\mathcal{F}^2_0\right]\right]=\E\left[f(X_0)\E\left[M_T\,|\,\mathcal{F}^2_0\right]\right]=\E\left[f(X_0)\right].
\]
Thus, the law of $X_0$ in the new probability space is the same.
This puts us in the framework of \cite{DFPR} and we can now proceed as in the proof of \cite[Th.\,13]{DFPR}.
\end{proof}

We provide some estimates for the laws $\mathcal{L}(X_{t}^{w})$ of the solution to \eqref{MKV} that will be used afterwards.
\begin{proposition}\label{lemma:DD}
Let Assumptions \ref{ass:A} and \ref{ass:w} hold and let $X^{w}_{t}$ be the unique weak-mild solution to \eqref{MKV}.
\begin{enumerate}[(i)]
\item There exists $c_{0}$ independent of {$\|w\|_{\infty}$} such that $$\displaystyle{\int_{H} |x|^{4} \mathcal{L}(X^{w}_{t})(\d x)\leq {c_{0}\left(1+\int_{H}|x|^4\,m_0(\d x)+\|w\|_{\infty}^4\right)}}, \ \ \ \forall t\in[0,T].$$
\medskip
\item For each $R>0$, there exists  a modulus of continuity $\omega_{m_{0},{R}}$ such that, if  $|w|_{\infty}\leq R$, it holds
$$\mathbf{d}_{1}(\mathcal{L}(X^{w}_{t}),
\mathcal{L}(X^{w}_{s}))\leq \omega_{m_{0},{R}}(|t-s|), \ \ \ \forall s, t\in[0,T].$$
\end{enumerate}
\end{proposition}
\begin{proof}
\begin{enumerate}[(i)]
\item
Using the definition of mild solution, the fact that $|e^{tA}|\leq 1$ and the estimate on the stochastic convolution $W_t^A$  \cite[Prop.\,1.144]{FGS}, we have, for some $c_{1}>0$,
\begin{eqnarray*}
\E[|X^{w}_{t}|^4]\le &
\E\left[
\left|e^{tA}X_{0}+\int_{0}^{t} e^{(t-s)A} w(s,X_{s}^{w})\d s + W^{A}_{t}\right|^4\right]\\
&\le
c_1(\E[|X_{0}|^{4}]+T^4 \|w\|_{\infty}^{4} +{1}), \ \ \ \forall \,0\leq t\leq T,
\end{eqnarray*}
which implies (i).
\item Let  $0\leq s\leq t\leq T$. Using the definition of mild solution, we have the existence of some $c_{1}>0$, independent of $s,t$  such that
\begin{align*}
&\mathbf{d}_{1}(\mathcal{L}(X^{w}_{t}),\mathcal{L}(X^{w}_{s})) \leq \mathsf{E}\left[|X^{w}_{t}-X^{w}_{s}|\right]
\le c_{1}\Bigg( \E\left[ \left|(e^{(t-s) A}-I)X_{0}\right|\right]\\&+ \E\left[\left|(e^{(t-s)A}-I) \int_{0}^{s} e^{(s-r)A} w(r,X_{r}^{w})\d {r}\right|+\int_{s}^{t} \left|e^{(t-r)A}w(r,X_{r}^{w})\right|\d {r}\right]\\&+  \E\left[|W^{A}_{t}-W^{A}_{s}|\right]\Bigg).
\end{align*}
We proceed to estimate the terms in the right hand side.
\medskip
\begin{enumerate}
\item Recalling that $|e^{\alpha A}|\leq 1$ for $\alpha\geq 0$, we have, for every $0\leq s\leq t\leq T$,
 \begin{equation*}\label{lala}
\E\left[\int_{s}^{t} \left|e^{(t-r)A}w(r,X_{r}^{w})\right|\d {r}\right]\leq (t-s) \|w\|_{\infty}.
\end{equation*}
\item Considering that $(e^{tA})_{t\in[0,T]}$ is strongly continuous and using dominated convergence, we have
$$
\E\left[ \left|(e^{(t-s) A}-I)X_{0}\right|\right]\to 0, \ \ \mbox{as} \ s\to t.
$$
Hence, there exists a modulus of continuity $\omega_{m_{0},A}$, independent of $w$ such that, for every $0\leq s\leq t\leq T$,
$$
\E\left[ \left|(e^{(t-s) A}-I)X_{0}\right|\right]\leq \omega_{m_{0},A}(|t-s|).
$$
\item
%

Let $R>0$. Since $A$ generates a semigroup of contractions, for every $0<\eta<s\leq T$ and $w$ such that $|w|_{\infty}\leq R$, we have ($\P\mbox{-a.s.}$)
$$\int_0^{s-\eta}e^{(s-r)A}w(r,X_{r}^{w})\d r\ \in \  e^{\eta A}(B_{TR}) =:K_{\eta,R}.$$
Assumption \ref{ass:A} guarantees that the semigroup is compact so $K_{\eta,R}$ is compact. 
 Considering also that the semigroup is strongly continuous, we have for $0\leq s\leq t\leq T$
 \begin{equation}\label{compact}
 \sup_{x\in K_{\eta,R}} |(e^{(t-s)A}-I)x|\leq \omega_{\eta,A,R}(t-s),
 \end{equation}
for some modulus of continuity  $\omega_{\eta,A,R}$. Using \eqref{compact}, it follows that,  for every $0\leq \eta \le s\leq t\leq T$
\begin{align*}&
\E \left[\left|(e^{(t-s)A}-I)\int_0^s e^{(s-r)A}w(r,X_r^{w})\d r\right|\right]\\&\leq \E\left[ \left|(e^{(t-s)A}-I)\int_0^{s-\eta} e^{(s-r)A}w(r,X_r^{w})\d r\right|+  \left|(e^{(t-s)A}-I)\int_{s-\eta}^s e^{(s-r)A}w(r,X_r^{w})\d r\right|\right]\\&\leq \omega_{\eta,A,R}(t-s)+2\eta R.
 \end{align*}
 Now, define
 $$
 \tilde\omega_{A,R}(r):=\inf_{0<\eta\leq T} \Big\{\omega_{\eta,A,R}(r)+ 2\eta R\Big\}, \ \ \ r>0.
 $$
 Then, $\tilde{\omega}_{A,R}(0^+)=0$ and $\tilde\omega_{A,R}$ is concave; so, it is a modulus of continuity. Taking now the infimum over $0<\eta\leq T$ in the inequality above, we get, for every $0\leq s\leq t\leq T$,
 $$\E\left[ \left|(e^{(t-s)A}-I)\int_0^s e^{(s-r)A}w(r,X_r^{w})\d r\right|\right] \leq \tilde{\omega}_{A,R}(|t-s|).$$
%
\item The map $[0,T]\to L^{2}(\Omega,\P)$, $t\mapsto W^{A}_{t}$ is continuous (see \cite[Th.\,5.2]{DaPratoZabczyk14}); hence, also the map   $[0,T]\to L^{1}(\Omega,\P)$,  $t\mapsto W^{A}_{t}$ is continuous.
Hence, since $[0,T]$ is compact, there exists a modulus of continuity $\hat\omega_{A}$ such that, for every $0\leq s\leq t\leq T$,
$$
\E\left[|W^{A}_{t}-W^{A}_{s}|\right]\leq \hat \omega_{A}(|t-s|).
$$
\end{enumerate}
We now combine all the estimates of points (a)--(d) above to obtain Claim (ii).
\end{enumerate}
\end{proof}

\subsection{Assumptions on $\mathcal{H}$ and $G$}
We make the following assumptions about $\mathcal{H}$ and $G$.
 \begin{assumption}
 \label{ass:HJB}
 \begin{enumerate}[(i)]
 \item[]
 \item
{The function $\mathcal{H}$ is continuous, the function $\mathcal{H}(\cdot,0,\mu_{0})$ is bounded for some $\mu_{0}\in\mathcal{P}_{1}(H)$, and there exists $C>0$ such that}
$$
|\mathcal{H}(x,p,\mu)-\mathcal{H}(x,p',\mu')|\leq C(|p-p'|+\mathbf{d}_{1}(\mu,\mu')), \ \ \ \  \forall x\in H, \ p,p'\in H, \ \mu,\mu'\in\mathcal{P}_1(H).
$$
\item $\mathcal{H}$ is differentiable with respect to the second variable, {$\mathcal{H}_{p}$ is continuous and bounded}, and there exists $C>0$ such that
$$
|\mathcal{H}_{p}(x,p,\mu)-\mathcal{H}_{p}(x,p',\mu')|\leq C(|p-p'|+\mathbf{d}_{1}(\mu,\mu')), \ \ \ \ \forall x\in H, \ p,p'\in H, \ \mu,\mu'\in\mathcal{P}_1(H).
$$
\item
{$G$ is continuous and such that the function $\mu\mapsto G(\cdot,\mu)$ is continuous as a map from $\mathcal{P}_{1}(H)$ to $C_b(H)$.}
\end{enumerate}
\end{assumption}
\begin{remark}
Requiring Lipschitz continuity of $\mathcal{H}$ and $\mathcal{H}_{p}$ with respect to $p$ and $\mu$ variables is a little restrictive. However, Assumption \ref{ass:HJB} still covers many basic cases, especially those having separated structure. An example of $\mathcal{H}$ coming from a stochastic optimal control problem is provided in Section \ref{sec:example}.
\end{remark}

\section{The Fokker-Planck Equation: Weak Solutions}\label{sec:FP}
In this section, we study \eqref{FPbis}  when the term $\mathcal{H}_p(x,Dv(t,x),m(t))$ is frozen. So, given $w\in C_b((0,T)\times H;H$), we 
consider the linear Fokker-Planck equation in the space $\mathcal{P}_{1}(H)$
\begin{equation}\label{FP}
\begin{cases}
\partial_{t}m(t) - {L}^{*} m(t)-\mbox{div}(w(t,x)m(t))=0,\\
m(0)=m_{0}\in\mathcal{P}_1(H).
\end{cases}
\end{equation}
We will interpret \eqref{FP} in a weak sense,  {following the approach of \cite{BDPR}}. For that, 
{we define the set of test functions}
\begin{align}
\label{eq:DTdef}
\mathcal{D}_T=\Big\{&\varphi\in  C_b^{1,2}([0,T]\times H): \  A^{*}D\varphi(t,x) \in C_{b}([0,T]\times H;H),\nonumber\\
 &  \ D^{2}\varphi\in C_{b}([0,T]\times H;\mathcal{L}_{1}(H))\Big\}.
\end{align}
To provide a rigorous definition of solution to \eqref{FP}, we introduce the operator
\begin{equation}\label{a0}
(L_0 \phi)(x)= \langle x , A^{*}D\phi(x)\rangle + \frac{1}{2}\mbox{Tr}\,[D^2 \phi(x)].
\end{equation}
We notice that  $L_{0}\varphi(t,x):=[L_{0}\varphi(t,\cdot)](x)$ is well defined  for every $\varphi\in \mathcal{D}_{T}$ and $t\in[0,T]$.
The formal operators $L^*$ and ${\mathbf{div}}$ in \eqref{FP} are interpreted similarly as for the Fokker-Planck equation in the case $H=\R^{n}$ when we apply test functions in $\mathcal{D}_T$. The definition of a \emph{weak solution} to \eqref{FP} given below is {similar to} the one of \cite{BDPR}; however, we use a larger class of test functions {and we assume from the beginning that the solution is continuous in time.}
\begin{definition}\label{def:solFP}
A map {$m\in \mathcal{S}$}
is called a \emph{weak solution} to \eqref{FP} if $m(0)=m_{0}$ and
\begin{align}\label{weakFP}
&
\int_{H} \varphi (t,x) m(t,\d x) - \int_{H} \varphi (0,x) m(0,\d x)\\&
= \int_0^{t} \left(\int_{H}\big[\partial_{t}\varphi(s,x)+L_{0}\varphi(s,x)-\langle w(s,x), D\varphi(s,x)\rangle \big]m(s,\d x)\right)\d s\nonumber
\end{align}
for every $t\in[0,T]$  and $\varphi\in \mathcal{D}_T$.\footnote{{Instead of the requirement that $m\in \mathcal{S}$, in \cite{BDPR}  the notion of weak solution $m:[0,T]\to\mathcal{P}_1(H)$ requires that  $t\mapsto m(t)(I)$ is measurable for each $I\in\mathcal{B}(H)$; that is, in the terminology of \cite{BDPR},  $m$ is a probability kernel. Our requirement that $m\in\mathcal{S}$ is actually stronger. Indeed, given $I\in\mathcal{B}(H)$ and $m\in\mathcal{S}$, the map $t\mapsto m(t,I)$ is the composition of the continuous map $[0,T] \to (\mathcal{P}_1(H),\mathbf{d}_1), \ t\mapsto m(t)$ and of the map $(\mathcal{P}_1(H),\mathbf{d}_1)\to\R, \ \mu\mapsto \mu(I)$.
 The $\mathbf{d}_1$-topology is stronger than the weak topology of measures; so, to conclude that $m\in\mathcal{S}$ is a probability kernel, it suffices to show that $\mu\mapsto \mu(I)$ is measurable when $\mathcal{P}_{1}(H)$ is endowed with the Borel $\sigma$-algebra induced by the weak topology. This claim follows from the fact that the latter map is upper semicontinuous
for $I$ closed (that is standard, as a consequence of one of the equivalent definitions of weak convergence) and by \cite[Prop.\,7.25]{Bertsekas}. This makes our notion of solution stronger and enables us to invoke the uniqueness result of \cite{BDPR} in Proposition \ref{lemma:DF}.}} 

\end{definition}

The connection between \eqref{FP} and \eqref{MKV} is provided by the following.
\begin{proposition}\label{prp:w}
Let Assumptions \ref{ass:A} and \ref{ass:w}{(ii)} hold and let
 $X^{w}$ be the unique  weak-mild solution to \eqref{MKV} provided by Theorem \ref{th:mildSDE}. Then,
$\mathcal{L}(X^{w}_{\cdot})$ is a weak solution to \eqref{MKV}.
\end{proposition}
\begin{proof}
 This is just an application of Dynkin's formula for sufficiently regular functions: see, e.g., \cite[Prop.\,1.169]{FGS}, {applied there with 
 $$b_0(s,x,a_1)=a_1, \ a_1(s,\omega)=w(s,X^w_s(\omega)), \ a_2(\cdot)=0, \ \Lambda=B_{\|w\|_\infty}\subset H.
 \vspace{-.9cm}$$}
\end{proof}
Concerning  uniqueness of solutions to \eqref{FP}, we have the following.

\begin{proposition}\label{lemma:DF}
Let Assumptions \ref{ass:A} and  \ref{ass:w} hold. Then \eqref{FP} admits at most one weak solution in the class
 $$
\mathcal{S}_{4}:=\left\{m\in \mathcal{S}: \ {\int_0^T}\int_{H}|x|^{4}m(t,\d x){\d t}<\infty\right\}.
 $$
\end{proposition}
\begin{proof}
The claim follows from
\cite[Th.\,4.1]{BDPR} once we prove that a solution in our sense is also a solution in the sense of \cite[Th.\,4.1]{BDPR}.
To do that we first notice that, in
\cite[Formula (1.5)]{BDPR}, the definition of solution is given using the set test functions (called $\mathcal{E}_{A}(H)$ there)
which is the linear span of the real parts of the functions of the form
$$(t,x)\mapsto \phi(t)e^{\mathrm{i}\langle x,h(t)\rangle}, \ \ \ \phi\in C^{1}([0,T]), \ \ h\in C^{1}([0,T],D(A^{*})).$$
Given the above, we observe that such a set is clearly contained in the set $\mathcal{D}_{T}$ defined in \eqref{eq:DTdef}. Finally, we point out that in \cite{BDPR} the Kolmogorov operator $L$ is defined differently (corresponding to our $\partial_{t}+L_{0}-\langle w,D\rangle$); still, equality (1.5) in \cite{BDPR}  corresponds to our  \eqref{weakFP}.
\end{proof}

We combine and summarize the above two results in the following theorem.
\begin{theorem}
Let {Assumptions \ref{ass:A} and \ref{ass:w} hold.} Then \eqref{FP} admits a unique weak solution in the class $\mathcal{S}_{4}$, which coincides with the law of the unique weak-mild solution $X^{w}$ to \eqref{MKV}.
\end{theorem}
\begin{proof}
By Proposition \ref{lemma:DD}(i), our solution provided by Proposition \ref{prp:w} belongs to the class $\mathcal{S}_{4}$ and the uniqueness follows from Proposition \ref{lemma:DF}.
\end{proof}

\section{The HJB Equation: Mild Solutions}\label{sec:HJB}
In this section we deal with equation \eqref{HJBbis} for a given fixed $m\in\mathcal{S}$.
We introduce the concept of a \emph{mild solution} to \eqref{HJBbis} (see \cite[Section 4.4.1.2]{FGS}).
\begin{definition}[Mild solution to \eqref{HJBbis}]\label{def:mildHJB} Recall \eqref{eqB:Cmgammadefbis}.
A \emph{mild solution} to \eqref{HJBbis} is a function $v\in C_{b,\gamma}^{0,1}([0,T]\times H)$ such that, for all $(t,x)\in [0,T]\times H$,
$$
v(t,x)=[R_{T-t}\,G(\cdot,m(T))](x)-\int_{t}^T [R_{s-t}\mathcal{H}(\cdot,Dv(s,\cdot), m(s))](x)\d s.
$$
\end{definition}

\begin{theorem}\label{Th:HJB}
Let  $m\in \mathcal{S}$ and let Assumptions \ref{ass:A} and  \ref{ass:HJB} hold.
\begin{enumerate}[(i)]
\item  There exists a unique mild solution $u$ to \eqref{HJBbis}.
 \item If $m_{n}\to m$ in $\mathcal{S}$, denoting by $v_{n}$ and $v$ the corresponding mild solutions to \eqref{HJBbis}, we have   $$\|Dv_{n} -Dv\|_{C_{b,\gamma}}\to 0.$$
\end{enumerate}
\end{theorem}
\begin{proof}
\begin{enumerate}[(i)]
\item
This part is contained in \cite[Theorem 4.90]{FGS}.\footnote{More precisely, it is enough to take in that theorem, $m=0$, $\gamma_G(t)=t^{-\gamma}$, $G=\mathbf{Id}$. We warn that $G$ in this footnote refers to the operator used in \cite[Theorem 4.90]{FGS}, which has nothing to do with our function $G$.}
\item Using the smoothing properties of the Ornstein-Uhlenbeck semigroup (see e.g.
    \cite[Theorem 4.37]{FGS}), we have the representation of the gradients
$$
Dv(t,x)=D[R_{T-t}\,G(\cdot,m(T))](x)-\int_{t}^T D[R_{s-t}\mathcal{H}(\cdot,Dv(s,\cdot), m(s))](x)\d s.
$$
$$
Dv_{n}(t,x)=D[R_{T-t}\,G(\cdot,m_{n}(T))](x)-\int_{t}^T D[R_{s-t}\mathcal{H}(\cdot,Dv_{n}(s,\cdot), m_{n}(s))](x)\d s.
$$
Hence, fixing $t\in[0,T)$ and considering \cite[Th.\,4.37]{FGS} and  {Assumption \ref{ass:HJB}}, we have, for some $C>0$, 
\begin{align*}
&\|Dv(t,\cdot) -Dv_n(t,\cdot)\|_{\infty}\leq \left\|D[R_{T-t} (G(\cdot,m(T))-G(\cdot,m_{n}(T)))]\right\|_{\infty}\\&
\ \ \ \ + \int_{t}^T \Big\|D\Big[R_{s-t}\big(\mathcal{H}(\cdot,Dv(s,\cdot), m(s))- \mathcal{H}(\cdot,Dv_{n}(s,\cdot), m_{n}(s)\big)\Big]\Big\|_{\infty}\,\,\d s\\
&\leq C \Big((T-t)^{-\gamma} \|G(\cdot,m(T))-G(\cdot,m_{n}(T))\|_{\infty}\\
& \ \ \ \ \ \ \ \ \ \  +\int_{t}^T (s-t)^{-\gamma}\left(\|Dv(s,\cdot)-Dv_n(s,\cdot)\|_{\infty}+ \mathbf{d}_{1}\big(m(s),m_{n}(s)\big)\Big)\d s\right)\\
&\leq C\left({\eta_{n}}(T-t)^{-\gamma}+ {\rho_{\infty}\big(m,m_{n}\big)} T^{1-\gamma}+ \int_{t}^T (s-t)^{-\gamma}\|Dv(s,\cdot)-Dv_n(s,\cdot)\|_{\infty}\right)\d s,
\end{align*}
where, {by Assumption \ref{ass:HJB}(iii),}
$$
{\eta_{n}:=  \|G(\cdot,m(T))-G(\cdot,m_{n}(T))\|_{\infty}\to 0\quad\mbox{as}\,\,n\to\infty.}
$$
By a suitable generalization of Gronwall's Lemma (see \cite[Prop. D.30]{FGS}), we get
\begin{align*}
&\|Dv(t,\cdot) -Dv_n(t,\cdot)\|_{\infty}\leq C(T-t)^{-\gamma}\left({\eta_{n}}+  T  {\rho_{\infty}\big(m,m_{n}\big)}\right).
\end{align*}
Multiplying both sides by $(T-t)^{\gamma}$ and taking the supremum over $t\in[0,T)$, we get 
\begin{align*}
&\|Dv -Dv_n\|_{C_{b,\gamma}}\leq C\left({\eta_{n}}+  T  {\rho_{\infty}\big(m,m_{n}\big)}\right)
\end{align*}
and the claim follows.
\end{enumerate}
\end{proof}

\section{The MFG System: Existence of Solutions}\label{sec:MFGsystem}
We turn to the analysis of the coupled MFG system  \eqref{HJBbis}-\eqref{FPbis}. Let Assumptions \ref{ass:w}(i) and \ref{ass:HJB} hold. Given $m\in \mathcal{S}$, we denote by $v^{(m)}$
the unique mild solution to \eqref{HJBbis}.
Moreover, recalling the notation $X^{w}$ used in Section \ref{sec:FP}, we set 
$$X_t^{(m)}:= X_{t}^{\mathcal{H}_{p}(\cdot,Dv^{(m)}(\cdot,\cdot), m(\cdot))}.$$

In the remainder of this section, $\{\mathbf{e}_{k}\}_{k\in\mathbb{N}}$ is  the orthonormal basis provided by Assumption \ref{ass:A}.
\begin{lemma}\label{lemma:SS}
Let Assumptions \ref{ass:A} and  \ref{ass:w} holds. For $n\in \N$, set
$$\alpha_{n}:=(2|\lambda_{n}|)^{-1}, \ \ \ \beta_{n}:=\int_{H} \langle x, \mathbf{e}_n\rangle^{2} m_{0}(\d x).
$$

\begin{enumerate}[(i)]
\item
We have
\begin{equation*}\label{AN}
 \int_{H} \sup_{t\in[0,T]} \langle e^{tA}x,\mathbf{e}_n\rangle^{2}m_{0}(\d x)\leq \beta_{n}, \ \ \ \forall n\in\N.
\end{equation*}
\medskip
\item
We have
$$\sup_{t\in[0,T],\,x\in H}\int_{0}^t \langle e^{(t-s)A}w(s,x),\mathbf{e}_n\rangle^{2} \d s
\leq  \alpha_{n} \cdot
\sup_{s\in (0,T),\, x\in H} \langle w(s,x),\mathbf{e}_n\rangle^{2}
\ \leq \ \alpha_{n} {\|w\|^2_{\infty}},$$ 
for every  $n\in\N.$
\medskip
\item
We have
\begin{equation*}\label{CN}
\sup_{t\in[0,T]}\E \langle W^{A}_{t},\mathbf{e}_n\rangle^2\leq  \alpha_{n}, \ \ \ \ \ \ \ \forall n\in\N.
\end{equation*}
\end{enumerate}
\end{lemma}
\begin{proof}
\begin{enumerate}[(i)]
\item In fact,
\begin{align*}
&\int_{H}\sup_{t\in[0,T]} \langle e^{tA}x,\mathbf{e}_n\rangle^{2}m_{0}(\d x)=
\int_{H}\sup_{t\in[0,T]} \langle x,e^{tA}\mathbf{e}_n\rangle^{2}m_{0}(\d x)
\\\\&
=\int_{H} \left(\sup_{t\in[0,T]} e^{2\lambda_{n} t}\right) \langle x, \mathbf{e}_n\rangle^{2} m_{0}(\d x) \leq 
\int_{H} \langle x, \mathbf{e}_n\rangle^{2} m_{0}(\d x)=\beta_{n}.
 \end{align*}
\item  We compute,
\begin{align*}
&\sup_{t\in[0,T], \, x\in H}\ \int_{0}^t \langle e^{(t-s)A}w(s,x),\mathbf{e}_n\rangle^{2} \d s= \sup_{t\in[0,T], \, x\in H}\ \int_{0}^t \langle w(s,x),e^{(t-s)A}\mathbf{e}_n\rangle^{2} \d s
\\\\
&= \sup_{t\in[0,T], \, x\in H}\ \int_{0}^t e^{2(t-s)\lambda_{n}} \langle w(s,x),\mathbf{e}_n\rangle^{2} \d s\\& \leq
\sup_{s\in (0,T),\, x\in H} \langle w(s,x),\mathbf{e}_n\rangle^{2}\cdot
\sup_{t\in [0,T]}\int_{0}^t e^{2(t-s)\lambda_{n}}\d s
\\[3mm]
&\le
\alpha_n \cdot \sup_{s\in (0,T),\, x\in H} \langle w(s,x),\mathbf{e}_n\rangle^{2}
\le
\alpha_n  \|w\|^2_{\infty}.
\end{align*}
\item Recall that, by Assumption \ref{ass:A}, we have
$$\mbox{Cov}(W^{A}_{t})=Q_{t}=\int_{0}^{t} e^{2sA}\d s.$$ Then,
 $$
\sup_{t\in[0,T]}\E \langle W^{A}_{t},\mathbf{e}_n\rangle^2 =\sup_{t\in[0,T]} \langle Q_{t}\mathbf{e}_{n},\mathbf{e}_n\rangle
\leq \langle Q_{T}\mathbf{e}_{n},\mathbf{e}_n\rangle= (2|\lambda_{n}|)^{-1}.
$$
\end{enumerate}
\end{proof}
\bigskip
Set\footnote{{Note that $a_n$ only depends on the data  $A$, $\mathcal{H}$, and on the initial measure $m_0$.}}
$$a_{n}:=3\Big(\beta_{n}+\alpha_{n}+\alpha_{n}\|\mathcal{H}_{p}\|_\infty^2\Big).$$
Then, by Assumptions \ref{ass:A} and \ref{ass:w}(i),
we have
\begin{equation}\label{sumak}
\mathbf{a}=(a_{n})_{n\in\mathbb{N}}\in \ell_{1}.
\end{equation}
Let us
consider the map
\begin{equation}\label{Def:Psi}
\Psi:  \mathcal{S}\to \mathcal{S}, \ \ \ \Psi(m)(t):= \mathcal{L}(X_{t}^{(m)}).
\end{equation}
This map is well defined  due to Proposition \ref{lemma:DD}.

\begin{proposition}\label{prop:stima}
Let Assumptions \ref{ass:A},  \ref{ass:w}(i) and \ref{ass:HJB}  hold. Then, for each $m\in \mathcal{S}$, $n\in\N$, and $t\in[0,T]$, we have 
\begin{align*}
\int_{H} |\langle x,\mathbf{e}_n\rangle|^{2} \Psi(m)(t,\d x)\leq  a_{n}.
\end{align*}
\end{proposition}
\begin{proof}
We have, using the equation for $X_t^{(m)}$ and denoting $w(s,x)=\mathcal{H}_{p}(x,Dv^{(m)}(s,x),m(s))$, 
\begin{align*}
& \int_{H} |\langle x,\mathbf{e}_n\rangle|^{2} \Psi(m)(t,\d x)=
\E\langle X^{(m)}_{t}, \mathbf{e}_n\rangle^{2}\\
&\leq 3
\E\left[\langle e^{tA} X_{0},\mathbf{e}_n\rangle^{2} + \int_{0}^t \langle e^{(t-s)A} w(s,X^{(m)}_s),\mathbf{e}_n\rangle^{2} \d s + \langle W^{A}_{t},\mathbf{e}_n\rangle^{2}\right]
\end{align*}
Applying  Lemma \ref{lemma:SS}(i), we get the following estimate for the first term
$$
\E\left[\langle e^{tA} X_{0},\mathbf{e}_n\rangle^{2} \right]
=\int_H\langle e^{tA} x,\mathbf{e}_n\rangle^{2}m_0(\d x)
\le \beta_n.
$$
As for the second term, we have by Lemma \ref{lemma:SS}(ii)
\begin{align*}
\E\left[ \int_{0}^t \langle e^{(t-s)A} w(s,X^{(m)}_s),\mathbf{e}_n\rangle^{2} \d s\right] &
\le \sup_{x\in H}  \int_{0}^t \langle e^{(t-s)A} w(s,x),\mathbf{e}_n\rangle^{2} \d s \ \le \ \alpha_n \|\mathcal{H}_p\|_{\infty}^2.
\end{align*}
The last term is estimated by
Lemma \ref{lemma:SS}(iii), which gives
$$
\E\left[\langle W^{A}_{t},\mathbf{e}_n\rangle^{2}\right]
\le \alpha_n.
$$
The claim now follows from the definition of $a_n$.
\end{proof}

Let $c_0$ be the constant of Proposition \ref{lemma:DD}. We set 
$$
\hat{c}_{m_0}:=1+ c_0\left(1+\int_{H}|x|^4m_{0}(\d x)+
\|\mathcal{H}_{p}\|_{\infty}^4\right)
$$
and
\begin{equation}\label{eq:qo}
\mathcal{R}_{m_0}:=\left\{\mu\in\mathcal{P}_{1}(H): \  \int_{H} |x|^{4}\mu(\d x)\leq \hat{c}_{m_0}\right\}.
\end{equation}
Observe that,   by Lemma \ref{lemmaq} and by \cite[Remark 7.13 (ii)]{Villani},  the metric $\mathbf{d}_{1}$ metrizes on $\mathcal{R}_{m_0}$ the topology of weak convergence in $\mathcal{P}(H)$.
In view of Proposition \ref{prop:stima}, we also define
$$\mathcal{Q}_{m_{0}}:=
\left\{\mu\in \mathcal{R}_{m_0}: \ \int_{H}\langle x,\mathbf{e}_{k}\rangle^{2}\mu(\d x)\leq a_{k}\right\}.$$
By Proposition \ref{lemma:compactPH}, $\mathcal{Q}_{m_0}$ is $\mathbf{d}_{1}$-compact.
We now consider a subset of $\mathcal{S}$ of functions which take values in the set $\mathcal{Q}_{m_0}$ and are $\mathbf{d}_{1}$-equi-uniformly continuous. More precisely,
recalling Proposition \ref{lemma:DD}(ii) we define
\begin{equation}\label{eq:defomegam0}
\omega_{m_0}:=\omega_{m_0,\|\mathcal{H}_{p}\|_{\infty}}
\end{equation}
and consider the subset $\mathcal{C}_{m_0}\subset \mathcal{S}$ defined as
\begin{align}
\label{CM}
\mathcal{C}_{m_0}:=&\Big\{m:[0,T]\to \mathcal{Q}_{m_0}:  \  m(0)=m_{0},
\ \mathbf{d}_{1}(m(t),m(s))\leq \omega_{m_{0}}(|t-s|)
\Big\}.
\end{align}
\begin{lemma}\label{lemma:CC}
Let Assumptions \ref{ass:A}, \ref{ass:w}{(i)} and  \ref{ass:HJB} hold. The set
$\mathcal{C}_{m_0}$ is convex and compact in $(\mathcal{S},\rho_\infty)$.
\end{lemma}
\begin{proof}

\emph{Relative compactness.}
Relative compactness follows from
Proposition \ref{lemma:compactPH} and the Arzel\`a-Ascoli's theorem for functions with values in complete metric spaces.

\medskip

\emph{Closedness.}
Let $\{m_{n}\}_{n\in\mathbb{N}}\subset \mathcal{C}_{m_0}$ be such that $m_{n}\to m\in \mathcal{S}$. This means that  $$\sup_{[0,T]}\mathbf{d}_{1}(m_{n}(t), m(t))\to 0.$$ Clearly $m(0)=m_{0}$. Moreover, $m_{n}(t)$ converges weakly to $m(t)$ for all $t\in[0,T]$, so
$$
 \int_{H} (|x|^{4}\wedge K) m(t,\d x) =  \lim_{n\to\infty} \int_{H} (|x|^{4}\wedge K) m_n(t,\d x)\leq \hat{c}_{m_0}, \ \ \ \forall t\in[0,T], \ \forall K>0.
$$
Hence,
by dominated convergence
$$
 \int_{H} |x|^{4} m(t,\d x) = \lim_{K\to\infty}
 \int_{H} (|x|^{4}\wedge K) m(t,\d x) \leq \hat{c}_{m_0}.
$$
Similarly, one shows that
$$\sup_{t\in[0,T]} \int_{H}\langle x,\mathbf{e}_{k}\rangle^{2}m(t,\d x)\leq a_{k} \ \ \forall k\in\mathbb{N}.$$
Next, for $t,s\in[0,T]$, we have
\begin{align*}
 \mathbf{d}_{1}(m(t),m(s))& \leq \mathbf{d}_{1}(m_{n}(t), \mu(t))+ \mathbf{d}_{1} (m_{n}(s),m(s))+
 \mathbf{d}_{1}(m_{n}(t),m_{n}(s))		\\
 &\leq \mathbf{d}_{1}(m_{n}(t), m(t))+ \mathbf{d}_{1} (m_{n}(s),m(s)) +\omega_{m_{0}}(|t-s|)
\end{align*}
Taking the limit as $n\to\infty$ above, we get
$$ \mathbf{d}_{1}(m(t),m(s)) \leq \omega_{m_{0}}(|t-s|).$$
%

\medskip

\emph{Convexity.} Let $m_{1},m_{2}\in  \mathcal{C}_{m_0}$, $\lambda\in[0,1]$, and set $m_{\lambda}:=\lambda m_{1}+(1-\lambda)m_{2}$.
The only nontrivial claim to verify is that, for $t,s\in[0,T]$,
$$\mathbf{d}_{1}(m_{\lambda}(t),m_{\lambda}(s))\leq \omega_{m_{0}}(|t-s|).$$
We will use the fact (see, e.g.,  \cite[Rem.\,7.5(i)]{Villani} or  \cite[Cor.\,5.4]{CarmonaDelarueBook1}) that for $\mu,\nu\in \mathcal{P}_1(H)$ 
\[
\mathbf{d}_{1}(\mu,\nu)=\sup_{f\in \mbox{\footnotesize{Lip}}_1}\int_H f(x)(\mu(\d x)-\nu(\d x)),
\]
where ${\rm Lip}_1$ is the set of Lipschitz functions on $H$ with Lipschitz constant $1$.
Then
\begin{align*}
&\mathbf{d}_{1}(m_{\lambda}(t),m_{\lambda}(s))=\sup_{f\in \mbox{\footnotesize{Lip}}_1}\int_H f(x)(m_\lambda(t,\d x)-m_\lambda(s,\d x))
\\
&\leq
\lambda \sup_{f\in\mbox{\footnotesize{Lip}}_1}\int_H f(x)(m_1(t,\d x)-m_1(s,\d x))+(1-\lambda) \sup_{f\in \mbox{\footnotesize{Lip}}_1}\int_H f(x)(m_2(t,\d x)-m_2(s,\d x))
\\
&
=\lambda\mathbf{d}_{1}(m_1(t),m_1(s))+(1-\lambda)\mathbf{d}_{1}(m_2(t),m_2(s))\leq\omega_{m_{0}}(|t-s|),
\end{align*}
the claim.
\end{proof}
\medskip
In order to use a fixed point theorem to prove existence of solutions to the MFG system, we need to embed properly $\mathcal{C}_{m_0}$ into a topological vector space.
We consider  the  vector space $C([0,T];\mathcal{M}(H))$, where $\mathcal{M}(H)$ is endowed with the weak topology induced by the family of seminorms \eqref{semi}.
On the space $C([0,T];\mathcal{M}(H))$, we define a suitable topology  as follows. Define the family of seminorms
$$
\left|m(\cdot)\right|_{f}:=\sup_{t\in[0,T]}\left|\int_{H} f(x) m(t,\d x)\right|, \ \ \ \ f\in C_{b}(H).
$$
This family of seminorms induces on $C([0,T];\mathcal{M}(H))$ a topology, that we denote by $\tau_{\mathbf{uw}}$ and call  the \emph{uniform-weak topology},  that makes $C([0,T];\mathcal{M}(H))$  a locally convex topological vector space.
\begin{lemma}\label{lem:topology}
{Let Assumptions \ref{ass:A}, \ref{ass:w}{(i)} and  \ref{ass:HJB} hold.} We have the following inclusion:
$$\mathcal{C}_{m_0} \ \subset \ C([0,T];\mathcal{M}(H)).$$
Moreover, the topology induced by $\rho_{\infty}$ on $\mathcal{C}_{m_0}$ is the same as the one induced   by $\tau_{\mathbf{uw}}$.
\end{lemma}

\begin{proof}
In the following we use the notation $m(\cdot)$ to denote elements of $\mathcal{C}_{m_0}$ in order to avoid confusion.

We first prove the inclusion. Obviously, every function in $\mathcal{C}_{m_0}$ can also be seen as a function with values in $\mathcal{M}(H)$.
The inclusion then holds if we prove that
every function in $\mathcal{C}_{m_0}$ is continuous as a function from $[0,T]$ to $\mathcal{M}(H)$.
Indeed, let $m(\cdot)\in \mathcal{C}_{m_0}$;
then, for any $t_0\in [0,T]$ and any sequence
$t_n \to t_0$ one has
$\mathbf{d}_1(m(t_n),m(t_0)) \to 0$ as $n\to \infty$.
This implies that, in particular, $m(t_n)$ converges weakly to $m(t_0)$ when $n \to  \infty$, which completes the proof.
\begin{enumerate}[(i)]
 \item
\emph{$\tau_{\mathbf{uw}}$  stronger than $\rho_{\infty}$.} Let $\bar{m}(\cdot)\in\mathcal{C}_{m_0}$, $\varepsilon>0$ and consider the $\rho_\infty$-neighborhood
 $$
 \mathcal{U}(\bar{m}(\cdot)):=\left\{m(\cdot)\in\mathcal{C}_{m_0}: \ \sup_{t\in[0,T]}\mathbf{d}_{1}(m(t),\bar{m}(t))<\varepsilon\right\}.
 $$
 By the equiuniform $\mathbf{d}_{1}$-continuity of the elements of $\mathcal{C}_{m_0}$, there exists $\{t_{1},...,t_{N}\}\subset[0,T]$ such that
  $$\mathcal{U}_{N}(\bar{m}(\cdot)):=\Big\{m(\cdot)\in \mathcal{C}_{m_0}: \  \ \ \mathbf{d}_{1}(m(t_{i}),\bar{m}(t_{i}))<\varepsilon/2,\  \  \forall i=1,...,N \Big\}\subseteq \mathcal{U}(\bar{m}(\cdot)).$$
Since $\mathbf{d}_{1}$ metrizes the weak topology in $\mathcal{Q}_{m_0}$, for each $i=1,...,N$ we have the existence of $\delta>0$ and  a family $\{f_{i,1},...,f_{i,M}\}\subset C_{b}(H)$ such that
\begin{align*}
\mathcal{V}_{t_{i}}(\bar{m}(t_{i}))&:=\{\mu\in \mathcal{Q}_{m_0}: \ |\mu-\bar{m}(t_{i})|_{f_{i,j}}<\delta, \ \   \ \forall j=1,...,M\}\\
&\subseteq \{\mu\in \mathcal{Q}_{m_0}: \ \mathbf{d}_{1}(\mu,\bar{m}(t_{i}))<\varepsilon/2\}=:\mathcal{U}_{t_{i}}(\bar{m}(t_{i}))
\end{align*}
Define now the $\tau_{\mathbf{uw}}$-neighborhood  of $\bar{m}(\cdot)$
 $$
  \mathcal{V}(\bar{m}(\cdot)):=\left\{\mu\in\mathcal{C}_{m_0}: \ |m(\cdot)-\bar{m}(\cdot)|_{f_{i,j}}<\delta, \\\ \forall j=1,...,M, \ \forall i=1,...,N \right\}.$$
  Then,
\begin{align*}
 \mathcal{V}(\bar{m}(\cdot))&\subseteq \{m(\cdot)\in\mathcal{C}_{m_0}: \ m(t_{i})\in \mathcal{V}_{t_{i}}(\bar{m}(t_{i})) \ \ \forall i=1,...,N\}\\
 &\subseteq
\left\{m(\cdot)\in\mathcal{C}_{m_0}:  \ m(t_{i})\in \mathcal{U}_{t_{i}}(\bar{m}_{0}(t_{i})),  \ \forall i=1,...,N\right\}\\
&= \mathcal{U}_{N}(\bar{m}(\cdot)) \subseteq \mathcal{U}(\bar{m}(\cdot)),
\end{align*}
concluding the proof of this part.

\smallskip

 \item
\emph{$\rho_\infty$ stronger than $\tau_{\mathbf{uw}}$.}  First, let us show that, for each $f\in C_{b}(H)$, the family
\begin{equation}\label{qui}
 \left\{[0,T]\to \R, \ \ \ t\mapsto \int_{H}{f}(x)m(t,\d x)\right\}_{m(\cdot)\in\mathcal{C}_{m_0}}
\end{equation}
is equiuniformly continuous; that is, there exists a modulus of continuity $\omega_{f}$ such that
\begin{equation}\label{quibis}
\left|\int_{H}{f}(x)(m(t,\d x)-m(s,\d x))\right|\leq \omega_{f}(|t-s|), \ \ \forall t,s\in[0,T], \ \ \ \forall m\in \mathcal{C}_{m_0}.
\end{equation}
Assume, by contradiction, that this is not true. Then, we may find $\eta>0$ and suitable sequences $\{s_{n}\},\{t_{n}\}\subset[0,T]$,
$\{m_{n}(\cdot)\}\subset \mathcal{C}_{m_0}$,
such that
\begin{equation}\label{use}
\left|\int_{H}f(x)(m_{n}(t_{n},\d x)-m_{n}(s_{n},\d x))\right|\geq \eta,   \ \ |s_{n}-t_{n}|\to 0.
\end{equation}
Now, on the one hand, by $\mathbf{d}_{1}$-equiuniform continuity of the functions belonging to  $\mathcal{C}_{m_0}$, we have
\begin{equation}\label{df}
\mathbf{d}_{1}(m_{n}(s_{n}),m_{n}(t_n))\to 0;
\end{equation}
on the other hand,
 by weak compactness of $\mathcal{Q}_{m_0}$, up to passing to subsequences, we have
 \begin{equation}\label{gh}
 m_{n}(s_{n})\stackrel{w}{\to}{\mu_{1}}, \ \  m_{n}({t_{n}})\stackrel{w}{\to}{\mu_{2}}
 \end{equation}
 {for some $\mu_1,\mu_2\in \mathcal{Q}_{m_0}$.}
 By \eqref{df}, we must have
 \begin{equation}\label{def}
{\mu_{1}=\mu_{2}.}
 \end{equation}
  Then, letting $n\to\infty$ in  \eqref{use} and  using \eqref{gh} and \eqref{def}, we get
 $$0<\eta\leq \lim_{n\to \infty} \left|\int_{H}f(x)(m_{n}(t_{n},\d x)-m_{n}(s_{n},\d x))\right|= 0,$$
 a contradiction. This proves \eqref{qui}.\medskip

Now, let  $\bar{m}(\cdot)\in\mathcal{C}_{m_0}$, $\varepsilon>0$, $\{f_{1},...,f_{N}\}\subset  C_{b}(H)$ and consider the $\tau_{\mathbf{uw}}$-neighborhood
 $$
 \mathcal{V}(\bar{m}(\cdot)):=\Big\{m(\cdot)\in\mathcal{C}_{m_0}: \ |m(t)-\bar{m}(t)|_{f_{i}}<\varepsilon \ \ \ \forall t\in[0,T], \ \forall i=1,...,N \Big\}.
 $$
 By \eqref{qui}, we have the existence of $\{t_{1},...t_{M}\}\subset [0,T]$ such that
   \begin{align*}
 \mathcal{V}_{N}(\bar{m}(\cdot))&:=\Big\{m(\cdot)\in\mathcal{C}_{m_0}: \ |m(t_{j})-\bar{m}(t_{j})|_{f_{i}}<\varepsilon/2 ,\ \ \forall i=1,...,N, \ \forall j=1,...,M \Big\}\\&\subseteq \mathcal{V}(\bar{m}(\cdot)).
 \end{align*}
 Since $\mathbf{d}_{1}$ metrizes the weak topology on $\mathcal{Q}_{m_0}$, for each $j=1,...,M$ we have the existence of $\delta_{j}>0$ such that
 $$\mathbf{d}_{1}(m(t_{j}),\bar{m}(t_{j})) <\delta_{j} \ \Rightarrow \ |m(t_{j})-\bar{m}(t_{j})|_{f_{i}}<\varepsilon/2, \ \ \ \ \ \forall i=1,...,N.$$
Let now $\delta:=\min\{\delta_{1},...,\delta_{M}\}>0$.
 Then,
 \begin{align*}
\mathcal{U}(\bar{m}(\cdot))&:=\left\{m(\cdot)\in \mathcal{C}_{m_0}: \ \sup_{t\in[0,T]} \mathbf{d}_{1}(m(t),\bar{m}(t))<\delta \right\}
\\
&\subseteq \Big\{m(\cdot)\in \mathcal{C}_{m_0}: \  \mathbf{d}_{1}(m(t_{j}),\bar{m}(t_{j}))<\delta \ \ \forall j=1,...,M \Big\}\\&\subseteq  \mathcal{V}_{N}(\bar{m}(\cdot))\subseteq  \mathcal{V}(\bar{m}(\cdot)),
\end{align*}
concluding the proof.
 \end{enumerate}
 \end{proof}
We now define a solution to the MFG system \eqref{HJBbis}--\eqref{FPbis}.
\begin{definition}\label{def:mfg}
A pair $(v,m)$ is a solution to the MFG system \eqref{HJBbis}--\eqref{FPbis} if $v$ is a mild solution to \eqref{HJBbis} (Definition \ref{def:mildHJB}) and  $m$ is a weak solution of \eqref{FP} with $w(t,x)=\mathcal{H}_{p}(x,Dv(t,x), m(t))$ (Definition \ref{def:solFP}).
\end{definition}
\begin{theorem}\label{teo:existence}
Let Assumptions \ref{ass:A}, \ref{ass:w}(i), and \ref{ass:HJB} hold.  Then, the MFG system \eqref{HJBbis}--\eqref{FPbis} admits a solution. 
\end{theorem}
\begin{proof}
By construction, if $m^{*}$ is a fixed point of the map $\Psi$ defined in \eqref{Def:Psi}, then the couple $(v^{(m^{*})},m^{*})$ is a solution to the MFG system \eqref{HJBbis}--\eqref{FPbis}. We are therefore going to show that $\Psi$ admits a fixed point.

We consider $\mathcal{C}_{m_0}$ embedded in the topological vector space $(C([0,T];\mathcal{M}(H)),\tau_{\mathbf{uw}})$. By 
Lemma \ref{lem:topology} we can indifferently use on $\mathcal{C}_{m_0}$ the $\tau_{\mathbf{uw}}$ or the $\rho_\infty$ topology.
In what follows, the topological notions are referred, indifferently, to one of those equivalent topologies.
%
%
%
We want to apply  Tikhonov's fixed point theorem.
First of all, we notice that, by the definitions of $\hat{c}_{m_0}$ and $\mathbf{a}$, using Proposition \ref{lemma:DD} and Proposition \ref{prop:stima}, we have
\begin{equation}\label{eq:pp}
\Psi(\mathcal{C}_{m_0})\subseteq \mathcal{C}_{m_0}.
\end{equation}
Since $\mathcal{C}_{m_0}$ is compact and convex, and since \eqref{eq:pp} holds, in order to conclude the existence of a fixed point of $\Psi$ in 
$\mathcal{C}_{m_0}$, we need to show that $\Psi|_{\mathcal{C}_{m_0}}$ is continuous.  The remaining part of the proof is devoted to this goal.

Let $\{m_{n}\}\subset\mathcal{C}_{m_0}$ be such that   $m_{n}\to m\in \mathcal{C}_{m_0}$.  Since $\mathcal{C}_{m_0}$ is compact, from each subsequence $(m_{n_{k}})$, one can extract a sub-subsequence such that $\Psi(m_{n_{k_h}})\to \hat{m}$ for some $\hat{m}\in \mathcal{C}_{m_0}$.
 Set
 $$w^{(m)}:=\mathcal{H}_{p}(\cdot,Dv^{(m)}(\cdot,\cdot), m(\cdot)), \ \ \ \ w^{(m_{n_{k_{h}}})}:=\mathcal{H}_{p}(\cdot,Dv^{(m_{n_{k_{h}}})}(\cdot,\cdot), m_{n_{k_{h}}}(\cdot)).$$
For each $h\in\N$,
\begin{align*}
&
\int_{H} \varphi (t,x) \Psi(m_{n_{k_{h}}})(t,\d x) - \int_{H} \varphi (0,x) m_0(\d x)\\&
= \int_0^{t} \left(\int_{H}\left[\partial_{t}\varphi(s,x)+L_{0}\varphi(s,x)-\langle w^{(m_{n_{k_{h}}})}(s,x), D\varphi(s,x)\rangle \right]\Psi(m_{n_{k_{h}}})(s,\d x)\right)\d s
\end{align*}
for every $t\in(0,T]$  and $\varphi\in D_{T}$.  
Considering that, by Theorem \ref{Th:HJB}(ii), we have $w^{(m_{n_{k_{h}}})}\to w^{(m)}$ pointwise as $h\to\infty$, we take the limit as $h\to\infty$ using dominated convergence theorem in the equation above to obtain
\begin{align*}
&
\int_{H} \varphi (t,x) \hat m(t,\d x) - \int_{H} \varphi (0,x) m_{0}(\d x)\\&
= \int_0^{t} \left(\int_{H}\left[\partial_{t}\varphi(s,x)+L_{0}\varphi(s,x)-\langle w^{(m)}(s,x), D\varphi(s,x)\rangle \right]\hat m(s,\d x)\right)\d s.
\end{align*}
On the other hand, we also have
\begin{align*}
&
\int_{H} \varphi (t,x) \mathcal{L}(X^{(m)}_{\cdot})(t,\d x) - \int_{H} \varphi (0,x) m_{0}(\d x)\\&
= \int_0^{t} \left(\int_{H}\left[\partial_{t}\varphi(s,x)+L_{0}\varphi(s,x)-\langle w^{(m)}(s,x), D\varphi(s,x)\rangle \right]\mathcal{L}(X^{(m)}_{\cdot})(s,\d x)\right)\d s.
\end{align*}
Hence, by Proposition \ref{lemma:DF}, we must have $\hat m=\mathcal{L}(X^{(m)}_{\cdot})$.
By the arbitrariness of $k\mapsto n_{k}$, it follows that 
$$
\Psi(m_{n})\to \hat m =\mathcal{L}(X^{(m)}_{\cdot})=\Psi(m) \ \ \ \mbox{in} \ (\mathcal{S},\rho_{\infty}),
$$
which gives the continuity of $\Psi$.
\end{proof}

\section{The MFG System: Uniqueness of Solutions}\label{sec:exist-uniq}
In this section, we establish the uniqueness of solutions for our MFG system \eqref{HJBbis}--\eqref{FPbis}. %
 It is well known, even in finite-dimensional state spaces (see \cite[Rem.\,1.4,\,p.32]{MR4214774}), that uniqueness of a solution for this system cannot generally be guaranteed. A typical approach to ensure uniqueness is to impose certain monotonicity conditions on the data.
One of the most common monotonicity conditions is the \emph{Lasry-Lions monotonicity condition} (see, for example, \cite{LasryLionsCRMI06, LasryLionsCRMII06, LasryLionsJJM07} and  \cite[Section 1.3.2]{MR4214774}). For other monotonicity conditions that ensure uniqueness in finite-dimensional MFG systems, readers can refer to \cite{Graber-Meszaros, MeszarosMou} and the references therein.
Another typical feature of uniqueness results is the separability of the Hamiltonian function $\mathcal{H}$ (as assumed in Assumption \ref{hp:uniqueness}(i) here). Well-posedness of MFG systems without separability has been studied in \cite{MeszarosMou}. Here, we assume separability and use a variant of the Lasry-Lions monotonicity condition from \cite[Theorem 1.4]{MR4214774}. Our assumptions are as follows.
\begin{assumption}
\label{hp:uniqueness}
\begin{enumerate}[(i)]\smallskip
\item[]
\smallskip
  \item
The Hamiltonian $\mathcal{H}$ has a separated form, i.e.
$$
\mathcal{H}(x,p,\mu)=\mathcal{H}^0(x,p)-F(x,\mu),
$$
for some continuous  functions
$$
\mathcal{H}^0:H \times H \to \R,
\qquad F: H\times \mathcal{P}_1(H)\to \R.
$$
\smallskip
  \item
The map $\mathcal{H}^0$ is convex in $p$.
\medskip
  \item
The functions $F$ and $G$ are monotone in $\mu\in\mathcal{P}_{1}(H)$ in the following sense
$$
\int_H [F(x,\mu_1)-F(x,\mu_2)](\mu_1-\mu_2)(\d x)
\ge 0
\quad
\forall \mu_1,\mu_2\in \mathcal{P}_1(H),
$$
$$
\int_H [G(x,\mu_1)-G(x,\mu_2)](\mu_1-\mu_2)(\d x) \ge 0
\quad
\forall \mu_1,\mu_2\in \mathcal{P}_1(H).
$$
\smallskip
  \item
One of the following condition is satisfied:
\begin{enumerate}[(a)]
\medskip
  \item The following is 
  satisfied:
$$
\int_H [F(x,\mu_1)-F(x,\mu_2)](\mu_1-\mu_2)(\d x)
> 0
\quad
\forall \mu_1,\mu_2\in \mathcal{P}_1(H),\;  \mu_1\ne \mu_2.
$$
\medskip
  \item {The following implications hold:}
$$
\int_H [F(x,\mu_1)-F(x,\mu_2)](\mu_1-\mu_2)(\d x)
= 0
\quad \Longrightarrow\quad F(\cdot,\mu_1)=F(\cdot,\mu_2),
$$
$$
\int_H [G(x,\mu_1)-G(x,\mu_2)](\mu_1-\mu_2)(\d x) = 0
\quad \Longrightarrow\quad G(\cdot,\mu_1)=G(\cdot,\mu_2).
$$
\medskip
  \item The following implication holds:
$$
\mathcal{H}^0(x,p_1)-\mathcal{H}^0(x,p_2)
-\langle \mathcal{H}^0_p(x,p_2), p_{1}-p_{2}\rangle =0
$$
$$
\quad \Longrightarrow\quad\mathcal{H}^0_p(x,p_1)=\mathcal{H}^0_p(x,p_2),
\quad \forall x\in H, \; p_1,p_2\in H
$$
\end{enumerate}
\end{enumerate}
\end{assumption}

\begin{theorem}
\label{th:uniqueness}
Let Assumptions  \ref{ass:w}(i), \ref{ass:HJB}, and
\ref{hp:uniqueness} hold.
Then, there exists a unique solution to
the MFG system  \eqref{HJBbis}--\eqref{FPbis}.
\end{theorem}
\begin{proof}
\emph{Existence.} Existence of a solution was proved in Theorem \ref{teo:existence}.

\smallskip
\emph{Uniqueness.} The proof is an adaptation to our  Hilbert space setting of a typical proof of uniqueness in finite dimensional spaces (see, e.g., \cite[Theorem 1.4]{MR4214774}). 

Let  $(v_1,m_1)$ and $(v_2,m_2)$ be two solutions
of our MFG system.
We set
$\bar{v}=v_1-v_2$ and $\bar{m}=m_1-m_2$.
Then, $\bar{v}$ solves the integral equation
\begin{align*}
\bar{v}(t,x)=&\ R_{T-t}[G(\cdot,m_1(T))-G(\cdot,m_2(T))](x)\\&
-\int_{t}^T R_{s-t}[\mathcal{H}^0(\cdot,Du_1(s,\cdot))
-\mathcal{H}^0(\cdot,Du_2(s,\cdot))](x)\d s\\&\ \ \
-\int_{t}^T R_{s-t}[F(\cdot,m_1(s))
-F(\cdot,m_2(s))](x)\d s,
\end{align*}
which means that $\bar{v}$ is a mild solution to the equation
$$
\begin{cases}
- {\partial_t}\bar{v}- {L} \bar{v}+
\mathcal{H}^0(x, Dv_1(t,x))-\mathcal{H}^0(x, Dv_2(t,x))
-\left( F(x,m_1(t))- F(x,m_2(t))\right)=0,\smallskip\\
\bar{v}(T,x)=G(x,m_1(T))-G(x,m_2(T)).
\end{cases}
$$
On the other hand, observing that that $\bar{m}_0\equiv 0$, we get that $\bar{m}$ satisfies, for every $\varphi\in\mathcal{D}_{T}$,
\begin{small}
\begin{align}\label{eq:FPbar}
&
\int_{H} \varphi (t,x) \bar{m}(t,\d x) - \int_{H} \varphi (0,x) \bar{m}(0,\d x)
= \int_0^{T} \int_{H}\left[\partial_{t}\varphi(s,x)+L_{0}\varphi(s,x)\right]\bar{m}(s,\d x)\\&
- \int_0^{T} \left(\int_{H}\langle\mathcal{H}^0_p(x,Dv_1(s,x)),\, D\varphi(s,x)\rangle \,\,m_1(s,\d x)
-\int_{H}\langle\mathcal{H}^0_p(x,Dv_2(s,x)),\, D\varphi (s,x)\rangle \,\,m_2(s,\d x)\right)\d s.\nonumber
\end{align}
\end{small}
We now set
$$
f(s,x):=
\mathcal{H}^0(x,Dv_1(s,x))
-\mathcal{H}^0(x,Dv_2(s,x))
-\left( F(x,m_1(s))- F(x,m_2(s))\right).
$$
By Assumption \ref{ass:HJB}, by Proposition \ref{lemma:DD}, and by the definition of solution to \eqref{HJBbis}-\eqref{FPbis}, we have  $f\in C_{b,\gamma}([0,T)\times H)$, where the latter space is defined in \eqref{eqB:Cmgammadefbis}.
Then, we consider the Kolmogorov equation
$$
\begin{cases}
{\partial_t}{v}(t,x)+ {L} {v}(t,x)=f(t,x),\medskip\\
v(T,x)=\phi(x):=G(x,m_1(T))-G(x,m_2(T)),
\end{cases}
$$
{By Theorem \ref{th:approx}, $\bar{v}$
is also a $\mathcal{K}$-strong solution of the same equation and moreover, we can choose 
the approximating data $\phi_n$, $f_n$ such that the solutions ${\bar{v}}_n$ of the approximating Kolmogorov equations 
$$
\begin{cases}
{\partial_t}{\bar{v}_n}(t,x)+ {L_0} {\bar{v}_n}(t,x)=f_n(t,x),\medskip\\
\bar {v}_n(T,x)=\phi_n(x)
\end{cases}
$$
belong to $\mathcal{D}_T$ for all $n\in \N$.}
Now,
using $\varphi= \bar{v}_{n}$ in  \eqref{eq:FPbar}, we get
\begin{small}
\begin{align}\label{aaa}
&
\int_{H} \bar{v}_n (T,x) \bar{m}(t,\d x)
=
\int_0^{T} \left(\int_{H}\left[\partial_{t}\bar{v}_n(s,x)+L_{0}\bar{v}_n(s,x) \right]\bar{m}(s,\d x)\right)\d s
\\&\nonumber
- \int_0^{T} \left(\int_{H}\left[\langle\mathcal{H}^0_p(x,Dv_1(s,x)),\, D\bar{v}_n(s,x)\rangle \right]m_1(s,\d x)
-\int_{H}\left[\langle\mathcal{H}^0_p(x,Dv_2(s,x)),\, D\bar{v}_n(s,x)\rangle \right]m_2(s,\d x)\right)\d s.
\end{align}
\end{small}
On the other hand, for every $s\in [0,T]$, we can integrate with respect to the measure $\bar{m}(s)$ the Kolmogorov equations solved by the $\bar{v}_n$'s, to get
$$
\int_{H}\left(\partial_{t} \bar{v}_{n}(s,x)+L_{0}\bar{v}_{n}(s,x)\right)\bar{m}(s,\d x)
=
\int_{H}f_n(s,x)
\bar{m}(s,\d x).
$$
Plugging this equality into \eqref{aaa}, we obtain
\begin{small}
\begin{align}\label{aaabis}
&
\int_{H} \bar{v}_n (T,x) \bar{m}(t,\d x) -\int_{0}^{T}\left(\int_{H}f_n(s,x)
\bar{m}(s,\d x)\right)\d s
=
\\&\nonumber
- \int_0^{T} \left(\int_{H}\left[\langle\mathcal{H}^0_p(x,Dv_1(s,x)), D\bar{v}_n(s,x)\rangle \right]m_1(s,\d x)
- \int_{H}\left[\langle\mathcal{H}^0_p(x,Dv_2(s,x)), D\bar{v}_n(s,x)\rangle \right]m_2(s,\d x)\right)\d s.
\end{align}
\end{small}
%
%

We now need to pass to the limit in \eqref{aaabis} as $n\to\infty$. By the definition of $\mathcal{K}$-convergence in $C_{b,\gamma}([0,T) \times H)$ (Definition \ref{dfB:piKconvweight}), there is $M\in\mathbb {R}$ such that for every $(s,x)\in [0,T)\times H$,
\[
\sup_{n\geq 1}|f_n(s,x)-f(s,x)|=(T-t)^{-\gamma}\sup_{n\geq 1}\Big[(T-s)^{\gamma}|f_n(s,x)-f(s,x)|\Big]\leq M(T-s)^{-\gamma}.
\]
Hence, since $(T-s)^{-\gamma}\in L^1(0,T)$ and $\lim_{n\to\infty}|f_n(s,x)-f(s,x)|=0$ for every $(s,x)\in [0,T)\times H$, we obtain
\[
\lim_{n\to\infty}\int_{0}^{T}\left(\int_{H}|f_n(s,x)-f(s,x)|
\bar{m}(s,\d x)\right)\d s=0
\]
by  Dominated Convergence Theorem. The convergence of the other terms follows directly by  Dominated Convergence Theorem, since the functions involved are uniformly bounded.
Therefore, passing to the limit in \eqref{aaabis}, we get
{\begin{small}
\begin{align}\label{aaabis2}
&0=-
\int_{H} \bar{v} (T,x) \bar{m}(t,\d x) +\int_{0}^{T}\left(\int_{H}f(s,x)
\bar{m}(s,\d x)\right)\d s
\\&\nonumber
- \int_0^{T} \left(\int_{H}\left[\langle\mathcal{H}^0_p(x,Dv_1(s,x)), D\bar{v}(s,x)\rangle \right]m_1(s,\d x)
- \int_{H}\left[\langle\mathcal{H}^0_p(x,Dv_2(s,x)), D\bar{v}(s,x)\rangle \right]m_2(s,\d x)\right)\d s.
\end{align}
\end{small}}
Hence, using the definition of $f$, we get
$$
0=\int_{H} \bar{u} (T,x) \bar{m}(t,\d x)+
\int_0^{T} \int_{H}[F(x,m_1(s))-F(x,m_2(s))]
\bar{m}(s,\d x)$$
$$
+
\int_0^{T} \int_{H}\left[ \mathcal{H}^{0}(x,Dv_2(s,x))-\mathcal{H}^{0}(x,Dv_1(s,x))
-\langle\mathcal{H}^0_p(x,Dv_1(s,x)),- D\bar v(s,x)\rangle \right]m_1(s,\d x)
$$
$$
+
\int_0^{T} \int_{H}\left[ \mathcal{H}^{0}(x,Dv_1(s,x))-\mathcal{H}^{0}(x,Dv_2(s,x))
-\langle\mathcal{H}^0_p(x,Dv_2(s,x)), \, D\bar{v}(s,x)\rangle \right]m_2(s,\d x).
$$
By Assumption \ref{hp:uniqueness}(ii)-(iii) all terms above are nonnegative, so they all must be equal to $0$.
We now conclude as follows.
\begin{enumerate}[(a)]
\item
If case (a) of Assumption \ref{hp:uniqueness}(iv) holds, we have $m_{1}=m_{2}$; then uniqueness of mild solutions of equation \eqref{HJBbis} gives $v_1=v_2$.
\item If case (b) of Assumption \ref{hp:uniqueness}(iv) holds, we first use uniqueness of mild solutions of equation \eqref{HJBbis}, which holds since 
$$F(x,m_1(s))=F(x,m_2(s)),  \ \ \ G(x,m_1(T))=G(x,m_2(T)),$$ to obtain $v_1=v_2$; then, we conclude that $m_{1}=m_{2}$ by uniqueness of weak solutions of \eqref{weakFP} with $w(s,x):=\mathcal{H}^0_p(x,Dv_1(s,x))=\mathcal{H}^0_p(x,Dv_2(s,x))$.
\item
{If case (c) of Assumption \ref{hp:uniqueness}(iv) holds, we first obtain 
$$\mathcal{H}^0_p(x,Dv_1(s,x))=\mathcal{H}^0_p(x,Dv_2(s,x)),$$ $m_1(s)$ and $m_2(s)$ a.e. for $s\in[0,T]$. This implies $m_{1}=m_{2}$ by uniqueness of weak solutions of \eqref{weakFP} with $w(s,x)=\mathcal{H}^0_p(x,Dv_1(s,x))$ (or equivalently with $w(s,x)=\mathcal{H}^0_p(x,Dv_2(s,x))$) and then conclude that $v_{1}=v_{2}$ using uniqueness of mild solutions of equation \eqref{HJBbis}.}
\end{enumerate}
\end{proof}

\section{Examples from Stochastic Optimal Control}\label{sec:example}
A typical application of our framework is when  the Hamiltonian $\mathcal{H}$ arises from stochastic optimal control problems. To illustrate that, consider the  $H$-valued stochastic optimal control problem with  value function
\begin{equation}\label{vex}
v(t,x)= \inf_{\alpha_\cdot \in \mathcal{A}_t} \E\left[\int_{t}^{T} f(X^{t,x,\alpha_\cdot}_{s},\alpha_{s}, m(s))\d s+g(X^{t,x,\alpha_\cdot}_{T},m(T))\right],
\end{equation}
where
$$\mathcal{A}_t=\big\{\alpha_\cdot:[t,T]\times \Omega \to\Lambda \ \ (\mathcal{F}_{s}^t)-\mbox{progressively measurable}\big\},$$
defined on some reference probability space $(\Omega,\mathcal{F}, (\mathcal{F}_s^t)_{s\geq t} ,\P,W)$, where $\Lambda$ is a complete separable metric space, $m(\cdot)\in\mathcal{S}$, and
$X^{t,x,\alpha_\cdot}$ solves, in mild sense  on $[t,T]$, the infinite dimensional SDE
\begin{equation}\label{SDEexample}
\d X_{s}=(A X_{s}+b(X_{s},\alpha_{s},m(s))) \d s+ \d W_{s}, \ \ \ \ X_{t}=x.
\end{equation}
We then have
$$
\mathcal{H}(x, p,\mu)=\sup_{\alpha\in \Lambda} \big\{-\langle b(x,\alpha,\mu),p\rangle-f(x,\alpha,\mu)\big\}.
$$
We make the following assumptions:
\begin{enumerate}[(i)]
\item
$\Lambda=B_R\subset H$;

\item
$b(x,\alpha, \mu)= -\alpha+b_0(x)$, where $b_0:H\to{H}$ is bounded and Lipschitz continuous;

\item
$f(x,\alpha,\mu)=f_0(x,\mu)+ f_1(|\alpha|)$, where $f_0:H\times\mathcal{P}_1(H)\to\R$ is continuous in all variables, $f_0(\cdot,\delta_0)$ is bounded, $f_0$ is Lipschitz continuous in $m$ with respect to ${\bf d}_1$, uniformly in $x$, and $f_1\in C^{1,1}_{\rm loc}(\mathbb{R})$ is symmetric and uniformly convex\footnote{This means that there exists $\eta >0$ such that $s\mapsto f_1(s) - \eta s^2$ is convex.};

\item
$g:H\times\mathcal{P}_1(H)\to\mathbb{R}$ is continuous and $g(\cdot,\delta_0)$ is bounded.
\end{enumerate}
Then,
\begin{align*}
\mathcal{H}(x, p,\mu)&=\sup_{|\alpha|\leq R}\big\{\langle \alpha,p\rangle-f_1(|\alpha|)\big\}-\langle b_0(x),p\rangle-f_0(x,\mu)\\&
=:\mathcal{H}^1(p)-\langle b_0(x),p\rangle-f_0(x,\mu).
\end{align*}
In this case
\[
\mathcal{H}^0(x,p)=\mathcal{H}^1(p)-\langle b_0(x),p\rangle,\quad F(x,\mu)=f_0(x,\mu),\quad G(x,\mu)=g(x,\mu),
\]
and it is easy to see that Assumption \ref{ass:HJB}(i) is satisfied. Regarding Assumption \ref{ass:HJB}(ii), we compute explicitly 
\[
\mathcal{H}^1(p)=\begin{cases}
(f_1')^{-1}(|p|)|p|-f_1((f_1')^{-1}(|p|)),\quad\mbox{if}\,\,|p|<f_1'(R),\medskip\\
R|p|-f_1(R),\ \ \ \ \ \ \ \ \ \ \ \ \ \ \ \ \ \ \  \ \ \ \ \quad\mbox{if}\,\,|p|\geq f_1'(R).
\end{cases}
\]
Since $f_1\in C^{1,1}(\mathbb{R})$ and is uniformly convex, we deduce that $(f_1')^{-1}$ is Lipschitz continuous, so $\mathcal{H}$ is Lipschitz continuous in $p$ too. Moreover, we have 
\[
D\mathcal{H}^1(p)=\begin{cases}
(f_1')^{-1}(|p|)\frac{p}{|p|},\ \ \quad\mbox{if}\,\,|p|<f_1'(R),\medskip\\
R\frac{p}{|p|}, \ \ \ \ \ \ \ \ \ \ \ \ \ \ \quad\mbox{if}\,\,|p|\geq f_1'(R),
\end{cases}
\]
and clearly this function is Lipschitz continuous. Hence, $\mathcal{H}^0_p(x,p)=D\mathcal{H}^1(p)-b_0(x)$ is bounded and Lipschitz continuous in $p$, uniformly in $x\in H$. {This makes our existence result, Theorem \ref{teo:existence}, applicable.}\\

As for uniqueness, examples of functions {$F$} satisfying Assumption \ref{hp:uniqueness}(iii) when $H=\R^d$ are for instance in Section 3.4.2 of \cite{carmona2018probabilistic}, volume I, and similar examples also work in a real separable Hilbert space $H$ and satisfy the conditions imposed in this paper. We present two examples.

{\begin{example}
Let $h_1:H\times\mathcal{P}_1(H)\to\mathbb{R}$ and $h_2:H\times\mathcal{P}_1(H)\to H$ be Lipschitz continuous and bounded. We define
\[
F_1(x,\mu):=h_1(x)\int_Hh_1(y)\mu(\d y),\quad F_2(x,\mu):=\left\langle h_2(x),\int_Hh_2(y)\mu(\d y)\right\rangle.
\]
We then have
\[
\int_H [F_i(x,\mu_1)-F_i(x,\mu_2)](\mu_1-\mu_2)(\d x)=\left|\int_Hh_i(y)(\mu_1-\mu_{2})(\d y)\right|^2,\quad i=1,2.
\]
Hence, Assumption \ref{hp:uniqueness}(iii) and Assumption \ref{hp:uniqueness}(iv)(b) are satisfied for $F_1$ and $F_2$. It is also easy to see that, for every $x\in H$,
\[
|F_i(x,\mu_1)-F_i(x,\mu_2)|\leq C_i \mathbf{d}_{1}(\mu_1,\mu_2),\quad i=1,2,
\]
where $C_i$ is the Lipschitz constant of $h_i$. Thus both functions satisfy Assumption \ref{ass:HJB}(i).
\end{example}
\begin{example}
Let $\ell:H\times[0,\infty)\to \R$ be bounded, continuous,  and  such that $\ell(x,\cdot)$ is strictly increasing for every $x\in H$ and is Lipschitz continuous with Lipschitz constant independent of $x$. Let $\rho:H\to[0,\infty)$ be bounded and Lipschitz continuous. Let $\nu$ be a positive finite measure on $H$ which is full\,\footnote{That is $\nu(D)>0$ for every open set $D$ in $H$. An example of a measure which is full is the nondegenerate Gaussian measure with mean $a\in H$ and covariance operator $Q\in \mathcal{L}_{1}^{+}(H)$, see \cite[Proposition 1.25]{DaPrato2006}  (nondegeneracy means that ${\rm ker}\,Q=\{0\}$).}. We define
\[
F(x,\mu):=\int_H \ell(z,\rho*\mu(z))\rho(z-x)\nu(dz),
\]
where $\rho*m$ denotes the convolution of $\rho$ with $\mu$, that is
\[
\rho*\mu(z)=\int_H \rho(z-u)\nu(du).
\]
We notice first that this convolution operation is Lipschitz continuous in $\mu$ and $z$. Indeed, let $z,y\in H$ and $\mu_1,\mu_2\in \mathcal{P}_1(H)$ and let $\gamma\in \Gamma(\mu_1,\mu_2)$. Then,
\[
\begin{split}
&|\rho*\mu_1(z)-\rho*\mu_2(y)|=\left|\int_H \rho(z-u)\mu_1(\d u)-\int_H \rho(y-w)\mu_2(\d w)\right|
\\
&\quad=\left|\int_H (\rho(z-u)-\rho(y-w))\gamma(\d u,\d w)\right|\leq C\int_H|(z-u)-(y-w)|\gamma(\d u,\d w)
\\
&
\quad\quad
\leq C|z-y|+C\int_H|u-w|\gamma(\d u,\d w).
\end{split}
\]
Taking the infimum over all $\gamma\in \Gamma(m_1,m_2)$ we thus obtain
\[
|\rho*\mu_1(z)-\rho*\mu_2(y)|\leq C(|z-y|+\mathbf{d}_{1}(\mu_1,\mu_2)).
\]
We now have
\[
\begin{split}
&\int_H [F(x,{\mu_1})-F(x,{\mu_2})]({\mu_1}-{\mu_2})(\d x)
\\
&
=\int_H [\ell(z,\rho*\mu_1(z))-\ell(z,\rho*\mu_2(z))]\int_H\rho(z-x)(\mu_1(\d x)-\mu_2(\d x))\nu(\d z)
\\
&
=\int_H [\ell(z,\rho*\mu_1(z))-\ell(z,\rho*\mu_2(z))][\rho*\mu_1(z)-\rho*\mu_2(z)]\nu(\d z)\geq 0,
\end{split}
\]
because $\ell(z,\cdot)$ is increasing. Moreover, if the last expression is equal to $0$, using that $\ell(z,\cdot)$ is strictly increasing, we conclude that 
$\rho*\mu_1(z)-\rho*\mu_2(z)=0, \mu$ a.e.. However, since this function is continuous and $\mu$ is full, this implies that $\rho*\mu_1(z)=\rho*\mu_2(z)$ for every $z\in H$. Hence, $F(\cdot,\mu_1)=F(\cdot,\mu_2)$ on $H$. Therefore, $F$ satisfies Assumption \ref{hp:uniqueness}(iii) and Assumption \ref{hp:uniqueness}(iv)(b).
Finally, we compute
\[
\begin{split}
&|F(x,\mu_1)-F(x,\mu_2)|
\leq C\int_H |\ell(z,\rho*\mu_1(z))-\ell(z,\rho*\mu_2(z))|\mu(\d z)
\\
&
\leq C\int_H |\rho*\mu_1(z)-\rho*\mu_2(z)|\mu(\d z)\leq C \mathbf{d}_{1}(\mu_1,\mu_2)
\end{split}
\]
which means that $F$ satisfies Assumption \ref{ass:HJB}(i).
\end{example}
\subsection{An economic model} 
We illustrate here an application to an economic model  in the presence of capital heterogeneity. Inspired  by the model presented in \cite[Sec.\,3(a)]{achdou}, a mean field game model of competition among firms operating in the same sector was studied in \cite{calvia2024mean}. Precisely, in \cite{calvia2024mean}, the production capacity of a representative firm is assumed  to follow the dynamics $$\d K_{s}= -\delta K_{s}\d s+\sigma K_{s}\d B_{s}+I_{s}\d {s}, \   \ K_0=k>0,$$ where $\delta,\sigma>0$,  $(I_{s})_{s\geq 0}$ is a nonnegative real valued process representing the investment rate of the firm, and $B$ is a standard one dimensional Brownian motion. According to the discussion in \cite[Sec.\,3(a)]{achdou},  the goal of a representative firm is to solve 
$$
 v(t,k):=\inf_{I_{\cdot}}\E\left[\int_{t}^{T} \left(- \pi(K_{s}, \mathcal{K}(m(s)))+\frac{1}{2} I_{s}^{2}\right)\d s\right],
$$
where $I_{\cdot}$ belong to a suitable set of admissible control processes,
\[
\mathcal{K}(\mu)= \left(\int_{\R} k^{\theta}\mu(\d k)\right)^{1/\theta} \ \ \ \  \ \ \mu \in \mathcal{P}_{\theta}(\R),
\] 
with $\theta >0$ and  $\pi:(0,\infty)^2\to \R$,  $(k,\mathcal{K})\mapsto \pi(k,\mathcal{K})$ is nondecreasing in $k$ and nonincreasing in $\mathcal{K}$.\footnote{In \cite{calvia2024mean}, the case $\theta=1$ and $\pi(k,\mathcal{K})=\mathcal{K}^{-\beta}k$, where $\beta >0$, is considered.}

This model can be modified as follows. First of all, 
in order to fit our setting, we assume that the noise is additive\footnote{
This is done to fit our additive noise setting. Note that, in this way, the capacity may become negative and this is of course an unpleasant fact from the point of view of applications.  It would thus be interesting to investigate, in future works, cases with non-additive noise which may ensure that the infinite dimensional process $X$ lies in the positive cone of $H=L^{2}(S^1;\R)$; see, e.g., \cite{GL}.}
 with the diffusion coefficient $1$; that is, we look at the dynamics of the form 
\[
\d K_{s}= -\delta K_{s}\d s+\d B_{s}+I_{s}\d {s}.
\] 
We may then  enrich the model assuming that, rather than having a large number of  single firms, we have a large number of networks of firms, each network being distributed around a circle $S^{1}$. The location of the firm in the representative network is parametrized by $a\in S^{1}$ and capital $K$ depends on this parameter. Therefore, the representative network's state is represented by a parametrized stochastic process $K_{s}(a)$, where $s\in[t,T]$ and $a\in S^{1}$. We may assume, following the modeling of \cite{BFFG}, that also the investment rate process $I_{s}(a)$ is parametrized by the location $a\in  S^{1}$, and that  capital $K_{s}(a)$ naturally evolves in time-space according to a depreciation law and,  by taking account of the ``spatial'' dimension $a\in S^{1}$, according  to a diffusion law; that is, according to the overall law 
\[
\frac{\partial^{2}K_{s}(a)}{\partial a^{2}} -\delta K_{s}(a).
\] 
Assuming also that the noise depends also on the location $a$, this  leads to the following controlled SPDE:
\begin{equation}\label{SDEeco}
\d K_{s}(a)=\left(\frac{\partial^{2}K_{s}(a)}{\partial a^{2}} -\delta K_{s}(a)\right)\d s+ \d B_{s}(a)+ I_{s}(a)\d s, \ \ \ K_0(a)=k(a).
\end{equation}
We consider the following generalization of the optimization problem above: 
\begin{equation}\label{funceco}
v(t,k(\cdot))=\inf_{I_{\cdot}(\cdot)}\E\left[\int_{t}^{T}\left( \int_{S^{1}}\left(-\pi(K_{s}(a), \mathcal{K}(m(s)))+\frac{1}{2} I_{s}(a)^{2}\right)\d a\right)\d s\right],
\end{equation}
where
\begin{equation}\label{Kgen}
\mathcal{K}(\mu)= \int_{H} \left(\int_{S_{1}}f(k(a))\d a\right)\mu(\d k), \ \ \ \mu\in \mathcal{P}_{1}(H),
\end{equation}
where $f:\R\to\R$ is a given Lipschitz and bounded function.
Formally setting  $W_{s}=B_{s}(\cdot)$ and assuming that $W$ is a cylindrical Brownian motion in the Hilbert space $H$, and formally setting $x=k(\cdot)$,  $X_{s}=K_{s}(\cdot)$, $\alpha_{s}=-I_{s}(\cdot)$, the SPDE \eqref{SDEeco} can be viewed in the setting of the present section as a controlled SDE \eqref{SDEexample} in the Hilbert space $H=L^{2}(S^{1};\R)$ with the following specifications:  
\[
A=\frac{\partial^{2}}{\partial a^{2}}-\delta,  \ \ \ b(x,\alpha)=- \alpha.
\]
Similarly, the optimal control problem can be reformulated as \eqref{vex} by setting $g=0$ and 
\[
f(x,\alpha, \mu)= \frac{|\alpha|^{2}}{2}+ F(x,\mu), 
\]
where
\[
F(x,\mu)=-\int_{S^{1}}\pi(x(a),\mathcal{K}(\mu))\d a.
\]
Imposing the constraint $I_{t}(\cdot)\in \Lambda=B_R$ and assuming that $\pi$ is Lipschitz and bounded, the conditions for our existence result (Theorem \ref{teo:existence}) are fulfilled.

Regarding uniqueness, we assume that $\pi(k,\mathcal{K})=\pi_{1}(\mathcal{K}) f(k)$, where $\pi_{1}$ is Lipschitz continuous, nonnegative, bounded, and strictly decreasing, and $f$ is the same function appearing in \eqref{Kgen} and required now to be  Lispschitz continuous, nonnegative, nondecreasing,  and bounded.    
Then, for every $\mu_1,\mu_{2}\in\mathcal{P}_{1}(H)$, we have 
\begin{align*}
&\int_H [F(x,\mu_1)-F(x,\mu_2)](\mu_1-\mu_2)(\d x)\\
=&-\int_{H}\left[\int_{S^{1}}\big(\pi(x(a),\mathcal{K}(\mu_{1}))- \pi(x(a),\mathcal{K}(\mu_{2}))\big)\d a\right](\mu_{1}-\mu_{2})(\d x)\\
 = &-
\big(\pi_{1}(\mathcal{K}(\mu_{1}))- \pi_{1}(\mathcal{K}(\mu_{2}))\big) \int_{H}\left(\int_{S^{1}}f(x(a)) \d a \right)(\mu_{1}-\mu_{2})(\d x)\\
 = &-
\big(\pi_{1}(\mathcal{K}(\mu_{1}))- \pi_{1}(\mathcal{K}(\mu_{2}))\big) (\mathcal{K}(\mu_{1})- \mathcal{K}(\mu_{2}))\geq 0.\\
 \end{align*}
 with equality if and only if $\mathcal{K}(\mu_{1})= \mathcal{K}(\mu_{2})$. This shows  that: 
 \begin{enumerate}[(i)]
\item   Assumption \ref{hp:uniqueness}(iii) is fulfilled; 
\item
 Since $$\big(\pi_{1}(\mathcal{K}(\mu_{1}))- \pi_{1}(\mathcal{K}(\mu_{2}))\big) (\mathcal{K}(\mu_{1})- \mathcal{K}(\mu_{2}))=0 \ \Longrightarrow \ \mathcal{K}(\mu_{1})= \mathcal{K}(\mu_{2}),$$ 
 also 
Assumption \ref{hp:uniqueness}(iv)(b) holds.
\end{enumerate}
 The remaining requirements of Assumption \ref{hp:uniqueness} are fulfilled by construction. Hence, our uniqueness result, Theorem \ref{th:uniqueness}, applies.

\appendix
\section{Compactness in the space of probability measures}\label{app:A}
We provide a result about compactness of sets in spaces of probability measures on $H$ which is used in the paper.
\begin{lemma}\label{lemmaq}
Let $L>0$ and
$$\mathcal{Q}_{L}:=\left\{\mu\in\mathcal{P}_{1}(H): \  \int_{H} |x|^{2}\mu(\d x)\leq L\right\}.$$
Then
$$
\lim_{R\to\infty}\sup_{\mu\in\mathcal{Q}_{L}}\int_{B_{R}^{c}} |x|\mu(\d x)\to 0.
$$
\end{lemma}
\begin{proof}
Since
$$
\sup_{\mu\in \mathcal{Q}_{L}}\int_{B_{R}^{c}} |x|\mu(\d x)\leq  \sup_{\mu\in\mathcal{Q}_{L}} \frac{1}{R} \int_{B_{R}^{c}} |x|^{2}\mu(\d x)\leq	 \frac{1}{R}\sup_{\mu\in\mathcal{Q}_{L}} \int_{H} |x|^{2}\mu(\d x)= \frac{L}{R},
$$
and the claim follows.
\end{proof}

\begin{proposition}\label{lemma:compactPH}
Let $S\subset \mathcal{P}_1(H)$. If, for some constant $c_1>0$, we have
$$
\sup_{\mu\in S} \int_{H} |x|^{2}\mu(\d x)\leq c_{1} \ \ \mbox{and} \ \
\lim_{N\to\infty}\sup_{\mu\in S}\sum_{i=N}^{\infty}\int_{H}\langle x,\mathbf{e}_i\rangle^{2}\mu(\d x)=0,
$$
then $S$ is relatively compact with respect to $\mathbf{d}_{1}$.
\end{proposition}
\begin{proof}
By
\cite[Ch.\,VI, Th.\,2.2]{Parthasarathy67}\footnote{{In \cite[Ch.\,VI, Th.\,2.2]{Parthasarathy67} it is stated that the second condition alone guarantees relative compactness of $S$. This is clearly not true, as can be seen by simply taking as $S$ a non compact sequence concentrated on the line generated by $\mathbf{e}_1$. However the proof of this theorem there remains valid if we add the assumption of uniform boundedness of second moments.}}, the set $S$ is tight. 
The claim follows using  \cite[Prop.\,7.1.5]{AGS}.
\end{proof}

\section{Classical and strong solutions to Kolmogorov equations}\label{app:B}
In this section, we define classical, mild, and $\mathcal{K}$-strong solutions of backward Kolmogorov equations and recall an approximation result  used in Section \ref{sec:exist-uniq}.
The material is mainly taken from \cite[Section B.7]{FGS} and is presented, for the reader's convenience, in a simplified form needed here.
We first consider the following terminal value problem (Kolmogorov equation)
in the Hilbert space $H$:
\begin{equation}
\label{eqB:Kbasic}
\left\{ \begin{array}{l}
\displaystyle{ \partial_t u (t,x)+L_0u(t,x)  =
 f(t,x),
\quad (t,x) \in [0,T) \times H,}\\
\\
u(T,x)=\varphi  (x), \ \ \ \ \  x \in H,
\end{array} \right.
\end{equation}
where $L_0$ is given in \eqref{a0}, {$\varphi \in C_b(H)$ and
$f\in C_{b,\gamma}([0,T)\times H)$ (recall that  $C_{b,\gamma}([0,T)\times H)$ was  defined in \eqref{eqB:Cmgammadefbis}).
Since, as recalled in a footnote in Section \ref{sec:FP}, the closure (in a suitable weak sense) of the operator $L_0$ is the generator (again in a suitable weak sense) of the semigroup
$R_t$ in \eqref{semigroup}},
we can formally rewrite \eqref{eqB:Kbasic} in the following mild form:
\begin{equation}\label{eqB:Kmild}
u(t,x)= R_{T-t}[\varphi](x) + \int_t^T R_{s-t}[f(s,\cdot)](x)\d s.
\end{equation}
We call the function $u$ above, defined by the right hand side, the {\em mild solution} of \eqref{eqB:Kbasic}.
Recalling the definition of $L_0$ in \eqref{a0}, we define  
\begin{align}\label{core}
D(L_0)&=\Big\{\phi \in UC^{2}_{b}(H): \ D^2\phi\in UC_b(H;\mathcal{L}_1(H)),   \  A^*D\phi\in UC_b(H;H)\Big\}.
\end{align}
and endow $D(L_0)$ with the norm
\begin{equation}
\label{eqB:hatA0norm}
\|\phi\|_{D(L_0)} := \|\phi\|_{0}+\|D\phi\|_{0}
+\|A^{\ast}D\phi\|_{0}+
\sup_{x \in H}\|D^2\phi(x)\|_{\mathcal{L}_{1}(H)}.
\end{equation}
Arguing as in Theorem 2.7 of \cite{DaPratoZabczyk95},
it can be proved that $D(L_0)$, endowed with the above
norm, is a Banach space\footnote{The definition of the domain of the operator studied in this paper
is slightly different, but the arguments of the mentioned reference can be
adapted easily. We also notice that the Banach space structure is not essential for our purposes, even if it simplifies the notation.}.
We now recall a notion
of  classical solution to \eqref{eqB:Kbasic} (see \cite[Definition B.82]{FGS})\footnote{We also observe that there are other definitions of classical solutions in the literature, see e.g. Section 6.2 of \cite[Section 6.2]{DaPratoZabczyk02},  \cite[Definition 4.6]{Gozzi97}, \cite[Definition 4.1]{Gozzi98}.
We refer interested readers to \cite[Section B.7.1]{FGS} for a detailed overview  on that.}. 
\begin{definition}
\label{dfB:KclassicalhatA0} We say that
$u \in C_{b}([0,T] \times H)$ is a \emph{classical solution to \eqref{eqB:Kbasic}
in $D(L_0)$} if
\begin{equation}
\label{eqB:regsolKclassicalhatA0}
\begin{cases}
u(\cdot, x) \in C^1([0,T]),\ \ \forall x \in H,
\\[2mm]
u(t, \cdot) \in  D(L_0)\;\; \text{for any $t\in [0,T]$}\
\; and \;\;\sup_{t \in [0,T]}\|u(t, \cdot)\|_{D({L}_{0})}<\infty,
            \\[2mm]
Du,  \,A^*Du\in C_{b}([0,T] \times H,H), \; \; D^2 u \in  C_{b}([0,T] \times H,\mathcal{L}_{1}(H))
          \end{cases}
\end{equation}
       and $u$ satisfies \eqref{eqB:Kbasic}  for every $(t,x) \in [0,T)\times H$.
\end{definition}
We stress that Definition \ref{dfB:KclassicalhatA0} implicitly implies that $\varphi \in D(L_0)$ and $f\in C([0,T]\times H)$.

{To introduce a notion of a strong solution, called $\mathcal{K}-$strong solution, we first recall a special case of the definition
of $\mathcal{K}-$convergence which is needed in this paper
(see \cite[Definition B.84]{FGS}).} 

\begin{definition}\label{dfB:piKconvweight}
Let $Z$ be a real Hilbert space.
Given  $(f_n)_{n \in \N}{\subset}
    C_{b,\gamma}([0,T) \times H,Z) $, we say that $f_{n}$ is 
$\mathcal{K}$-convergent to
    $f \in C_{b,\gamma}[0,T) \times H,Z)$
    if
    \begin{equation}
\label{eqB:Kconvweight}
\left\{
\begin{array} {l}
\displaystyle{\sup_{n \in \mathbb{N}}
\|f_{n}\|_{C_{b,\gamma}[0,T) \times H,Z)} <  \infty,}
\\[6mm]
\displaystyle{\lim_{n \to \infty } \sup_{(t,x)
        \in (0,T]\times K} (T-t)^{\gamma}|f_{n}(t, x) - f (t,x) | = 0,}
\end{array}
 \right.
\end{equation}
for every compact set $K \subset H$.
In this case, we write $\mathcal{K}-\lim_{n\to  \infty} f_n=f$ in
$C_{b,\gamma}[0,T) \times H,Z)$.
\end{definition}

{The definition of a $\mathcal{K}$-strong solution to \eqref{eqB:Kbasic} provided below is a special case of a more general definition that can be found in \cite[Definition B.85]{FGS}.}

\begin{definition}
\label{dfB:Kstrong}
Let $\gamma \in (0,1)$. Let $\varphi \in C_b(H)$ and
$f\in C_{b,\gamma}([0,T) \times H)$.
We say that a function
$u\in C_b([0,T] \times H)$
is a $\mathcal{K}$-strong solution
in $D(L_0)$ of (\ref{eqB:Kbasic})
if $u(t,\cdot)$ is Fr\'echet differentiable for each $t\in [0,T)$ and
there exist sequences
$(u_{n})_{n\in\mathbb{N}}\subset C_{b}([0,T] \times H)$,
$(\varphi_{n})_{n\in\mathbb{N}}\subset D(L_0)$,
$(f_{n})_{n\in\mathbb{N}}\subset C_{b,\gamma} ([0,T) \times H)$ such that:
\smallskip
\begin{itemize}
\item[(i)] For every $n \in \N$, $u_n$ is a classical solution in $D(L_0)$
(cf. Definition \ref{dfB:KclassicalhatA0}) to
\begin{equation}
\label{eqB:defs1}
\left\{
\begin{array}{l}
w_{t} =L_0 w  +f_{n}, \\
w(0) = \varphi_{n}.
\end{array}
\right.
\end{equation}
\item[(ii)]
The following limits hold
\[
\begin{cases}
\displaystyle{\mathcal{K}\mbox{--}\lim_{n\rightarrow\infty}\varphi_{n}=\varphi,  \qquad \qquad in \quad C_b(H),
}
\\[2mm]
\displaystyle{\mathcal{K}\mbox{--}\lim_{n\rightarrow\infty}u_{n}=u,
 \qquad \qquad in \quad C_b([0,T]\times H),
}
\end{cases}
\]
and
\[
\begin{cases}
\displaystyle{\mathcal{K}\mbox{--}\lim_{n\rightarrow\infty}
f_{n}=f, \qquad \quad \ \mbox{in}
\quad 
C_{b,\gamma} ([0,T) \times H)}
				\\[2mm]
				\displaystyle{\mathcal{K}\mbox{--}\lim_{n\rightarrow\infty}
Du_{n} =Du
\qquad \mbox{in} \quad 
C_{b,\gamma} ([0,T) \times H,H).
				}
	\end{cases}
\]
\end{itemize}
\end{definition}

We end this section with an approximation result which is a stronger version of \cite[Theorem B.95(i)]{FGS}.

\begin{theorem}
\label{th:approx}
$\phi \in C_b(H)$ and
$f\in C_{b,\gamma}([0,T)\times H)$.
Then the mild solution $u$ (defined in \eqref{eqB:Kmild}) of equation \eqref{eqB:Kbasic},
is also a $\mathcal{K}$-strong solution of to \eqref{eqB:Kbasic}. Moreover the sequence $(u_{n})$ can be chosen so that $u_n\in \mathcal{D}_T$ for all $n\in \N$.
\end{theorem}
\begin{proof}
We can apply
\cite[Theorem B.95\,(i)]{FGS}
(in the case when, in the notation there, $m=0$, $\eta(t)=t^{-\gamma}$), to
conclude that  $u$
is also a $\mathcal{K}$-strong solution of the same equation.
Moreover, from the proof of
\cite[Theorem B.95\,(i)]{FGS}
one can easily see that the approximating data $\varphi_n$, $f_n$, and consequently the approximating solutions $u_n$, can be chosen so that $\varphi_n \in D(L_0)$ and $f_n, \partial_{t}\bar{v}_{n}\in C_b([0,T]\times H)$ for all $n\in \N$.
As a consequence, we get that $u_n \in \mathcal{D}_T$ for all $n\in \N$.
\end{proof}

    \subsubsection*{\bf Acknowledgements} {\footnotesize{Salvatore Federico is a member of the GNAMPA-INdAM group and of the PRIN project 2022BEMMLZ ``Stochastic control and games and the
role of information'', financed by the Italian Ministery of University and Research.
Fausto Gozzi has been supported by the Italian
Ministry of University and Research (MIUR), in the framework of PRIN
projects 2017FKHBA8 001 (The Time-Space Evolution of Economic Activities: Mathematical Models and Empirical Applications) and 20223PNJ8K (Impact of the Human Activities on the Environment and
Economic Decision Making in a Heterogeneous Setting:
Mathematical Models and Policy Implications).
 }}

\end{document}